\documentclass[final,reqno,onefignum,onetabnum]{siamltex}
\usepackage{enumerate}
\usepackage{amsmath,amssymb,graphicx,bm}
\usepackage{verbatim}
\usepackage{mathrsfs,fancyhdr}
\usepackage[ruled,vlined]{algorithm2e}

\newtheorem{Lemma}{\indent Lemma}[section]
\newtheorem{Condition}{\indent Condition}[section]
\newtheorem{Theorem}{\indent Theorem}[section]
\newtheorem{Remark}{\indent Remark}[section]

\renewenvironment{proof}{{ \emph{Proof.} }}{\hfill $\Box$ \\}

\title{{Stochastic collocation methods via minimization of Transformed $L_1$ penalty }  \thanks{The work of Ling Guo and Yongle Liu was
 supported by the NSF of China (No.11671265) and Program for Outstanding Academic leaders in Shanghai City (No.151503100). The work of Jing Li was partially supported by the U.S. DOE ASCR project "Uncertainty Quantification For Complex Systems Described by Stochastic Partial Differential Equations".}}
\author{ Ling Guo \thanks{Department of Mathematics, Shanghai Normal University, Shanghai, China (lguo@shnu.edu.cn).}
\and
Jing Li \thanks{Pacific Northwest National Laboratory, Richland, Washington, 99354, USA (jing.li@pnnl.gov).}
\and
Yongle Liu \thanks{Department of Mathematics, Southern University of Science and Technology, Shenzhen, China (11749318@mail.sustc.edu.cn).}
           }
\begin{document}
\maketitle

\graphicspath{{polynomial_results/}{analytical_results/}}

\begin{abstract}
We study the properties of sparse reconstruction of transformed $\ell_1$ (TL1) minimization and present improved theoretical results about the recoverability and the accuracy of this reconstruction from undersampled measurements. We then combine this method with the stochastic collocation approach to identify the coefficients of sparse orthogonal polynomial expansions for uncertainty quantification. In order to implement the TL1 minimization, we use the DCA-TL1 algorithm which was introduced by Zhang and Xin. In particular, when recover non-sparse functions, we adopt an adaptive DCA-TL1 method to guarantee the sparest solutions. Various numerical examples, including sparse polynomial functions recovery and non-sparse analytical functions recovery are presented to demonstrate the recoverability and efficiency of this novel method and its potential for problems of practical interests.
\end{abstract}

\begin{keywords}
Uncertainty quantification, Stochastic collocation, DCA-TL1 minimization, Compressive sensing, Restricted isometry property

\end{keywords}


\pagestyle{myheadings}
\thispagestyle{plain}

\section{Introduction}
This paper is concerned with pursuing the sparse approximation of a target function $f:\mathbb{R}^d\rightarrow\mathbb{R}, d\geq1$, given relatively limited training data in the frame of the generalized polynomial chaos (gPC) \cite{Xiu_2002Wiener} approach. This problem stems from the field of uncertainty quantification (UQ) when the stochastic system is large-scaled and thus only limited simulations are affordable. Therefore the number of training samples is much less than the cardinality of the gPC basis. This results in an underdetermined system with an infinite number of solutions. Rooted in the idea of compressive sensing \cite{Candes_2006nos, Candes_2006Stablesrec, Donoho_2006Cs, Donoho_2006srs}, stochastic collocation methods via $\ell_1$ minimization have been shown to be efficient for recovering a sparse approximation of the target function. For more details and applications of collocation methods via $\ell_1$ minimization,  one can see \cite{Doostan_2011nonadapted,Yan_2012Sc,Rauhut_2012sparseLegen,mathelin_gallivan_2012,LiuGuo2016,Jakeman_2016generalizedsample,Chen2017,GuoNarayanZhou2018,XuZhou2018} and references therein.

Although $\ell_1$ minimization has attracted considerable attentions in compressive sensing, it may not perform well on some practical problems (when measurement matrix is coherent) and can not find the sparsest solutions. Recently, several nonconvex penalties have been developed to enhance the sparsity of the solutions since their metric, compared with the $\ell_1$ norm, are more `closer' to the $\ell_0$ norm.  Among them the $\ell_p$ penalty for $p\in (0,1)$ is the most popular one. The recoverability of $\ell_p$-minimization with $p\in(0,1)$ is established in \cite{Clayton2017}. Stochastic collocation method via $\ell_p$ minimization to recover the sparse gPC coefficients was explored in \cite{GuoLiu_2017Lq}.
A nonconvex panelty $\ell_{1-2}$ was proposed in \cite{Yin_2015L12} and \cite{Lou_2015ComputeL12} utilizing a difference of convex algorithm (DCA) \cite{Tao_1988DCoptim, Tao_1997convexanalysisdc} to obtain the sparse solution, and the numerical experiments demonstrated that $\ell_{1-2}$ is always better than $\ell_1$ in terms of sparsity. This $\ell_{1-2}$ minimization was further studied and combined with stochastic collocation for sparse gPC approximation in \cite{Yan_2015L12}.

In this paper, we focus on the nonconvex transformed $\ell_{1}$ penalty (TL1), which has been recently proposed and studied in \cite{Zhang_2016DCATL1}. There Zhang and Xin have shown that the TL1 penalty `interpolates' $\ell_{0}$ and $\ell_{1}$ norms through a nonnegative parameter $a\in (0,+\infty)$ and satisfies unbiasedness, sparsity and Lipschitz continuity and the DCA is proposed to compute TL1-minimizations. They have also conducted comparisons between  DCA-TL1 with DCA on other non-convex penalties such as PiE \cite{NguyenPiC2015}, MCP \cite{zhang2010}, and SCAD \cite{Fan_2001Variableselec},  and shown that DCA-TL1 performs competitively well. Based on all these facts we implement TL1 minimization for stochastic collocation.

The main contribution of this paper lies in the following two aspects. Firstly, based on the work in \cite{Candes_2008Rip} and \cite{Zhang_2016DCATL1}, we present improved theoretical results about the recoverability and the accuracy of reconstruction via TL1 minimization. These estimates cannot verify the superiority of the TL1 penalty over the $\ell_{1}$ penalty, just as the theoretical finding of the $\ell_{1-2}$ minimization in \cite{Yan_2015L12}. Secondly, we consider the stochastic collocation method via TL1 minimization to identify the gPC expansion coefficients. In particular, we focus on the Legendre polynomials in multi-dimensional case. In order to implement the constrained minimization problem with TL1 penalty, we employ the DCA-TL1 algorithm that introduced in \cite{Zhang_2016DCATL1} for sparse polynomial representations. For non-sparse functions, in order to pursue the sparest solution, we develop an adaptive approach, called the adaptive DCA-TL1 method. In particular, this algorithm allows several choices for parameter $a$. Given a set of values of $a$, for each fixed $a_i\in a$, we use the TL1 minimization algorithm to generate the minimizer, if the next solution obtained by setting $a_j\in a$ and $a_j\neq a_i$ is more sparse than the the previous one, then we select the better parameter $a_j$ and its corresponding more sparse solution. In this way, finally we will get the best parameter and the sparsest solution. Various numerical experiments demonstrate that the TL1 minimization method is
the most efficient and effective method among all three selected methods (the standard $\ell_{1}$, the $\ell_{1-2}$ and the TL1 minimization).

The rest of the paper is organized as follows. In Section 2, we introduce some preliminaries which are involved in the following parts and set up the TL1 minimization for stochastic collocation which is the main problem to be addressed in this paper. Section 3 is devoted to establish the new theoretical estimates for both sparse and non-sparse signals recovery, whether the noise is contained or not. The recover properties of stochastic collocation method using TL1 minimization are discussed in Section 4. Also, in this section, we introduce the DCA-TL1 algorithm and adaptive DCA-TL1 algorithm pseudo-code. Numerical experiments are then presented in Section 5, showing that the proposed approach outperforms the other two methods--the standard $\ell_{1}$ and the $\ell_{1-2}$ minimization. Concluding remarks are given in Section 6.

\section{gPC expansion and Problem setup}\label{sec:setup}
When a simulation model is computational expensive to run, building an approximation of the response of the model output with respect to the variations in the model input, is often an efficient approach to quantify parametric uncertainty. Polynomial approximation of a function (model) $f(z): \mathbb{R}^d\rightarrow\mathbb{R}, d\geq1$ where
$z= (z_1, z_2, \ldots, z_d)$ is a $d$-dimensional random variable with associated probability density function $w(z)$ is widely used for its fast convergence. Moreover, in this paper we will approximate $f$ with the gPC expansion.

The gPC expansion is an orthogonal polynomial approximation to random functions which has been broadly utilized in uncertainty quantification in recent decades \cite{Ghanem_1991Spectralappro, Xiu_2002Wiener,NarayanZhou2015,TangZhou2015}. Let $\alpha = (\alpha_1,\dots, \alpha_d)\in \mathbb{N}_0^{d}$ be a multi-index with $|\alpha| = \alpha_1+\dots+\alpha_d$, and $k\geq 0$ be an integer. Then the $k$th-degree gPC expansion $f_k(z)$ of function $f(z)$ is defined as,
\begin{equation}\label{eq:f_gPC_expan}
f(z) \approx f_k(z) \triangleq \sum_{\alpha \in \Lambda^d_{k}}x_{\alpha}\Phi_{\alpha}(z),\quad \Lambda^d_{k}= \left\{ \alpha\in \mathbb{N}_0^d : |\alpha| \leq k \right\}.
\end{equation}
where
$\Phi_\alpha(z)$ are the basis function satisfied
\begin{align}\label{eq:orthogonal}
\int\Phi_\alpha(z)\Phi_\beta(z)dw(z)=\delta_{\alpha\beta},
\end{align}
with $\delta_{\alpha\beta} \triangleq \prod_{i=1}^d \delta_{\alpha_i,\beta_i}$ to be the multi-index Kronecker delta. The probability density $w(z)$ in the orthogonality \eqref{eq:orthogonal} determines the type of orthogonal polynomials. In the following numerical examples, we primarily concentrate on the uniform distribution and the corresponding Legendre polynomials. For another types of orthogonal polynomials, one can see \cite{Xiu_2002Wiener} for details.
By placing an order for the orthogonal polynomials, we can change (\ref{eq:f_gPC_expan}) into the following single index version
\begin{align}
\label{eq:truncatePCE}
f_k(z)=\sum_{\alpha\in \Lambda^T_{k}}x_{\alpha}\Phi_{\alpha}(z)=\sum_{n=1}^{N}{x}_n \Phi_n(z).
\end{align}
where
\begin{align*}
N=  \# \Lambda^d_{k} \triangleq \dim T_k^d = \left( \begin{array}{c} d+k \\ k\end{array}\right).
\end{align*}

The goal is then to determine the expansion coefficients in this gPC type approximation \eqref{eq:truncatePCE}  when $f(z)$ has a sparse representation. Given a small set of $M\ll N$ unstructured realizations $\{z^i\}_{i=1}^M$ and the corresponding outputs $b=(f(z^1),...f(z^M))^{T}$, we seek an interpolation type sparse solution $x\in \mathbb{R}^N$ such that $f_k(z^i) = f(z^i)$ for $i = 1,\dots,M$ and meanwhile the number of nonzero entries of $x$ is the smallest. Namely, we solve the following optimization problem subject to a underdetermined linear system,
\begin{equation}\label{eq:gPC_L0}
\min \|x\|_0 \quad \text{subject to}\ Ax=b,
\end{equation}
where $x=(x_1,\ldots, x_N)^T \in \mathbb{R}^N$ is the coefficient vector in \eqref{eq:truncatePCE}, $\|x\|_0$ indicates the number of nonzero entries of  $x$ , and $A \in \mathbb{R}^{M \times N}$ usually called the measurement matrix is written as
 \begin{eqnarray}\label{eq:matrixelem}
A=(a_{ij})_{1\leq i\leq M,1\leq j\leq N}, \quad a_{ij}=\Phi_{j}(z^i).
 \end{eqnarray}
However, it is well know that this $\ell_0$ minimization problem \eqref{eq:gPC_L0} is a NP-hard problem \cite{Natarajan_1995L0}, we use the newly established constrained TL1 minimization \cite{Zhang_2016DCATL1} to get the gPC expansion coefficients $x$ by solving
\begin{align}\label{eq:gPC_TL1}
\min_{x\in\mathbb{R}^N}P_a(x) \quad \text{subject to}\ Ax=b,
\end{align}
where $P_a$ is some penalty function defined in \eqref{eq:P_a}, and we discuss in Section.\ref{sec:TL1} in details.

\section{Preliminaries}\label{sec:pre}
In order to present our method to solve problem \eqref{eq:gPC_TL1} we first briefly introduce some preliminaries of sparse recovery problem. From now on, we use $A$ to denote the interpolation matrix and $B$ to denote general sensing matrix (measurement matrix).
\subsection{Standard sparse recovery}
We consider the standard recovery problem of underdetermined system $Bx=b$, i.e.,
\begin{equation}\label{eq:L0min}
\min \|x\|_0 \quad \text{subject to}\ Bx=b,
\end{equation}
where the measurement matrix $B\in\mathbb{R}^{M\times N}$ with $M\ll N$. This $\ell_0$-minimization problem is NP-hard  in general (c.f. \cite{Natarajan_1995L0}). The well studied and popular approach to find sparse solutions to \eqref{eq:L0min} is compressive sensing (CS) \cite{Candes_2005error, Candes_2008Rip, Davies_2010RICLp,Donoho_2006srs}, where under certain circumstance, the solution or the approximation of the solution to \eqref{eq:L0min} can be found by the well studied $\ell_1$ minimization, which consists in finding the minimizer of the problem
\begin{equation}\label{eq:L1min}
\min \|x\|_1 \quad \text{subject to}\ Bx=b.
\end{equation}
This optimization problem can be solved with efficient algorithms from convex optimization (c.f. \cite{Berg_2007spgl}).

\subsection{Restricted isometry property} We recall the concept of restricted isometry constants (RIC), which characterizes matrices that are nearly orthonormal.
\begin{definition} \label{def:s-sparse-vec} A vector $x$ is said to be s-sparse if it has at most s nonzero entries, i.e., $x$ is supported on subset $T \subset \{1,\dots,N\}$ with $|T|\leq s$.
\end{definition}

\begin{definition}[restricted isometry constant\cite{Candes_2005Decoding, Candes_2006Stablesrec}]\label{def:RIC}
For each integer $s=1,2,\ldots,$ define the isometry constant $\delta_s$ of a matrix $B$ as the smallest number such that
\begin{equation*}
(1-\delta_s)\|x\|_2^2\leq\|Bx\|_2^2\leq(1+\delta_s)\|x\|_2^2
\end{equation*}
holds for any $s$-sparse vector $x$.
\end{definition}

Then we have the following lemma.
\begin{lemma}[c.f. \cite{Candes_2008Rip}]\label{lem:delta_2slem}
\begin{equation*}
|\langle Bx,Bx^{\prime}\rangle|\leq\delta_{s+s^{\prime}}\|x\|_2\|x^{\prime}\|_2.
\end{equation*}
for all $x$, $x^{\prime}$ supported on disjoint subsets $T,T^{\prime}\subseteq\{1,\ldots,N\}$ with $|T|\leq s,|T^{\prime}|\leq s^{\prime}$.
\end{lemma}
Follow the ideas of compressive sensing, we pursuit to approximate the solution to \eqref{eq:L0min} by solving the transformed $\ell_1$ minimization problem.
\subsection{Constrained transformed $\ell_1$ (TL1) minimization}\label{sec:TL1}
To obtain the TL1 minimization problem we first introduce the penalty function
\begin{equation}\label{eq:rho_a}
\rho_a(|t|)=\frac{(a+1)|t|}{a+|t|},
\end{equation}
and its properties which will be frequently involved later in the proofs.
\begin{Condition}[c.f. \cite{Zhang_2016DCATL1,Fan_2001Variableselec}]\label{condi:penaltyf}
The penalty function $\rho(\cdot)$ satisfies:
\begin{itemize}
  \item $\rho_a(t)$ is increasing and concave in $t\in[0,\infty]$ for all $a>0$;
  \item $\rho_a^{\prime}(t)$ is continuous with $\rho_a^{\prime}(0_{+})\in(0,\infty)$ for all $a>0$;
\end{itemize}
\end{Condition}

\begin{Lemma}\label{lem:normtrieq}
For all $a>0$,  $x_i,x_j \in \mathbb{R}$, the following inequalities hold:
\begin{equation*}
\rho_a(|x_i|)-\rho_a(|x_j|)\leq\rho_a(\big||x_i|-|x_j|\big|)\leq\rho_a(|x_i+x_j|)\leq\rho_a(|x_i|+|x_j|)\leq\rho_a(|x_i|)+\rho_a(|x_j|).
\end{equation*}
\end{Lemma}

\begin{proof}
The main idea is to use the monotonicity and concavity of $\rho_a(t)$.
The part,
\begin{equation}\label{eq:trianinequalitym}
\rho_a(\big||x_i|-|x_j|\big|)\leq\rho_a(|x_i+x_j|)\leq\rho_a(|x_i|+|x_j|)\leq\rho_a(|x_i|)+\rho_a(|x_j|),
\end{equation}
has been studied in \cite{Zhang_2016DCATL1} by Lemma 2.1.

To complete the proof, we have
\begin{equation}\label{eq:trianinequalityl}
\begin{split}
\rho_a(|x_i|)-\rho_a(|x_j|)&=\frac{(a+1)|x_i|}{a+|x_i|}-\frac{(a+1)|x_j|}{a+|x_j|}\\
&=\frac{(a+1)(|x_i|-|x_j|)}{a+|x_i|+|x_j|+|x_ix_j|/a}\\
&\leq\frac{(a+1)(|x_i|-|x_j|)}{a+\big||x_i|-|x_j|\big|}\\
&\leq\frac{(a+1)(\big||x_i|-|x_j|\big|)}{a+\big||x_i|-|x_j|\big|}=\rho_a(\big||x_i|-|x_j|\big|).
\end{split}
\end{equation}
Combining \eqref{eq:trianinequalitym} and \eqref{eq:trianinequalityl}, we get Lemma \ref{lem:normtrieq}.
\end{proof}

\begin{Remark}[\cite{Zhang_2016DCATL1}]
It worth noting that the penalty function $\rho_a$, $a>0$ acts almost like a norm. However, it lacks absolute scalability, i.e., $\rho_a(cx)\neq|c|\rho_a(x)$ in general.
\end{Remark}

We define the penalty function $P_a$ we used in our gPC coefficients recovery problem \eqref{eq:gPC_TL1} as
\begin{align}\label{eq:P_a}
P_a(x)=\sum_{i=1,\ldots,N}\rho_a(|x_i|),
\end{align}
where $\rho_a(|x|)$ is defined in \eqref{eq:rho_a}.
As discussed in \cite{Zhang_2016DCATL1}, the contour plots of the $P_a(x)$ penalty-metric in Figure \ref{fig:contours} show intuitively that when the parameter $a$ tends to zero, the level curves of TL1 approach the $x$ and $y$ axes (i.e., close to the $\ell_0$ norm), hence promoting sparsity. Meanwhile, when $a$ tends to $+\infty$, the level curves converge to those of $\ell_1$ norm.

\begin{Remark}
Note that, by the definition $P_a(x) =\sum_{i=1,\ldots,N}\rho_a(|x_i|) \leq \frac{a+1}{a}\|x\|_1$.
\end{Remark}

\begin{figure}[!ht]
\centering
\includegraphics[width=0.48\textwidth,height=0.23\textheight]{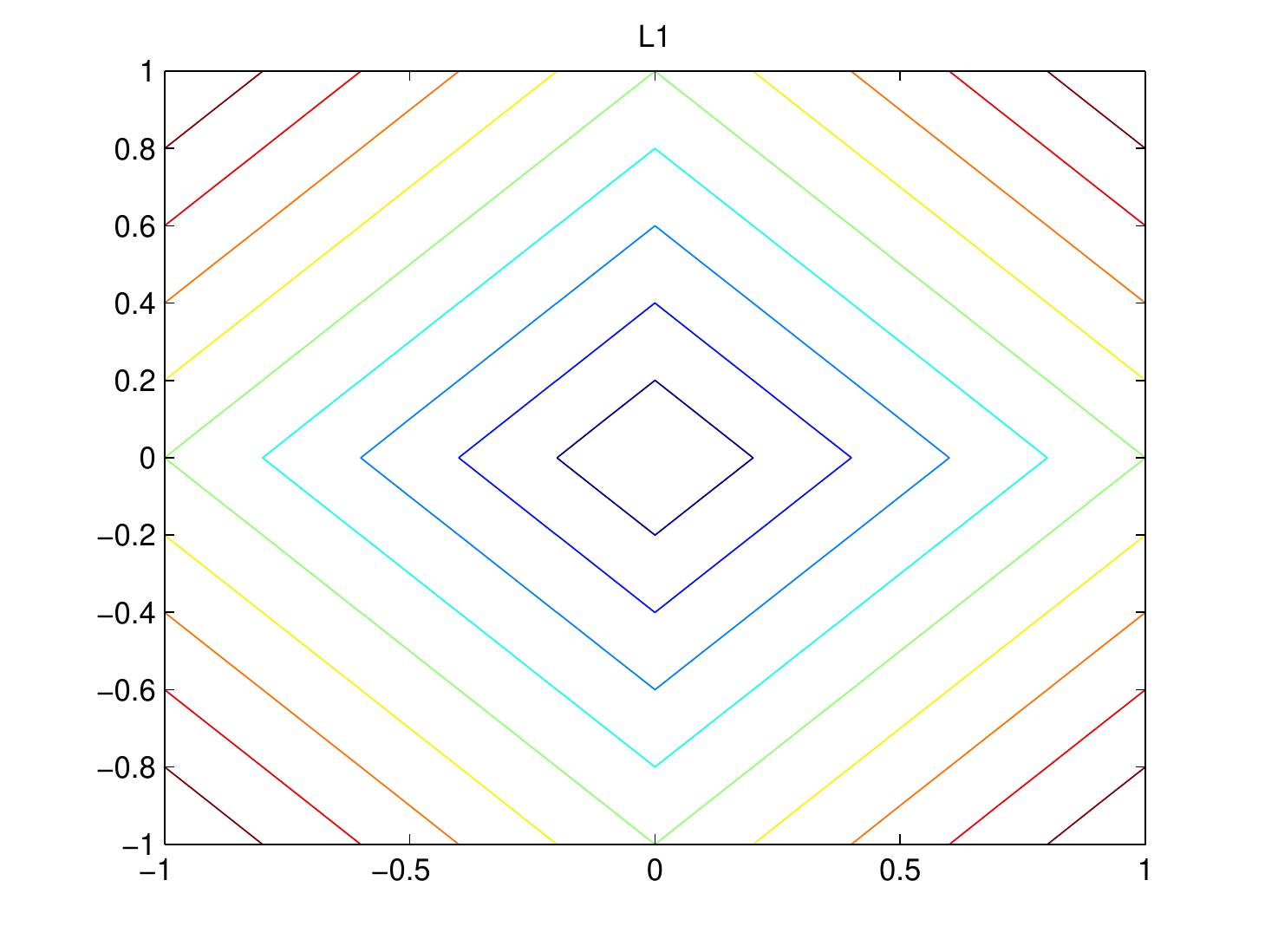}\quad
\includegraphics[width=0.48\textwidth,height=0.23\textheight]{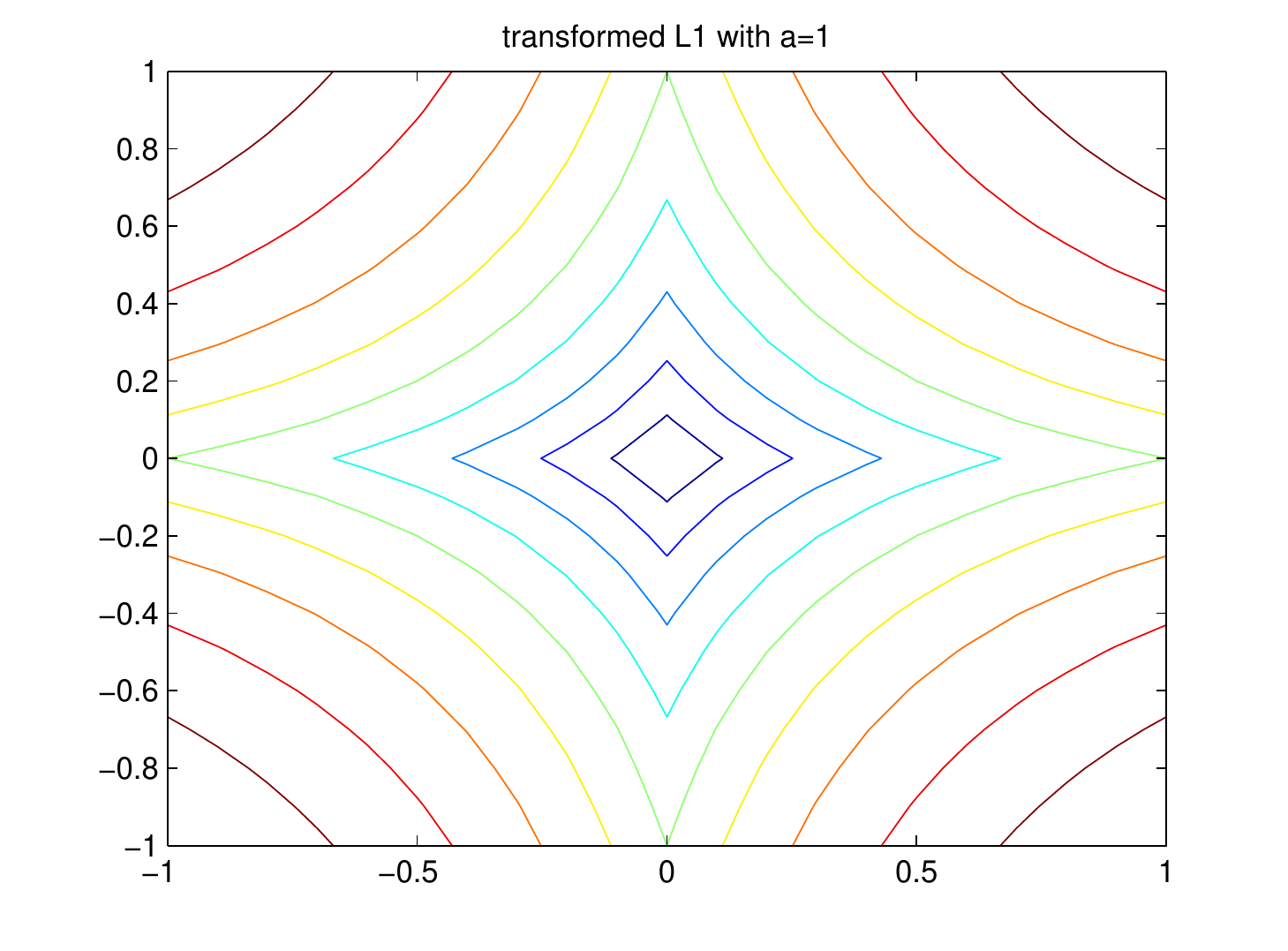}\\
\includegraphics[width=0.48\textwidth,height=0.23\textheight]{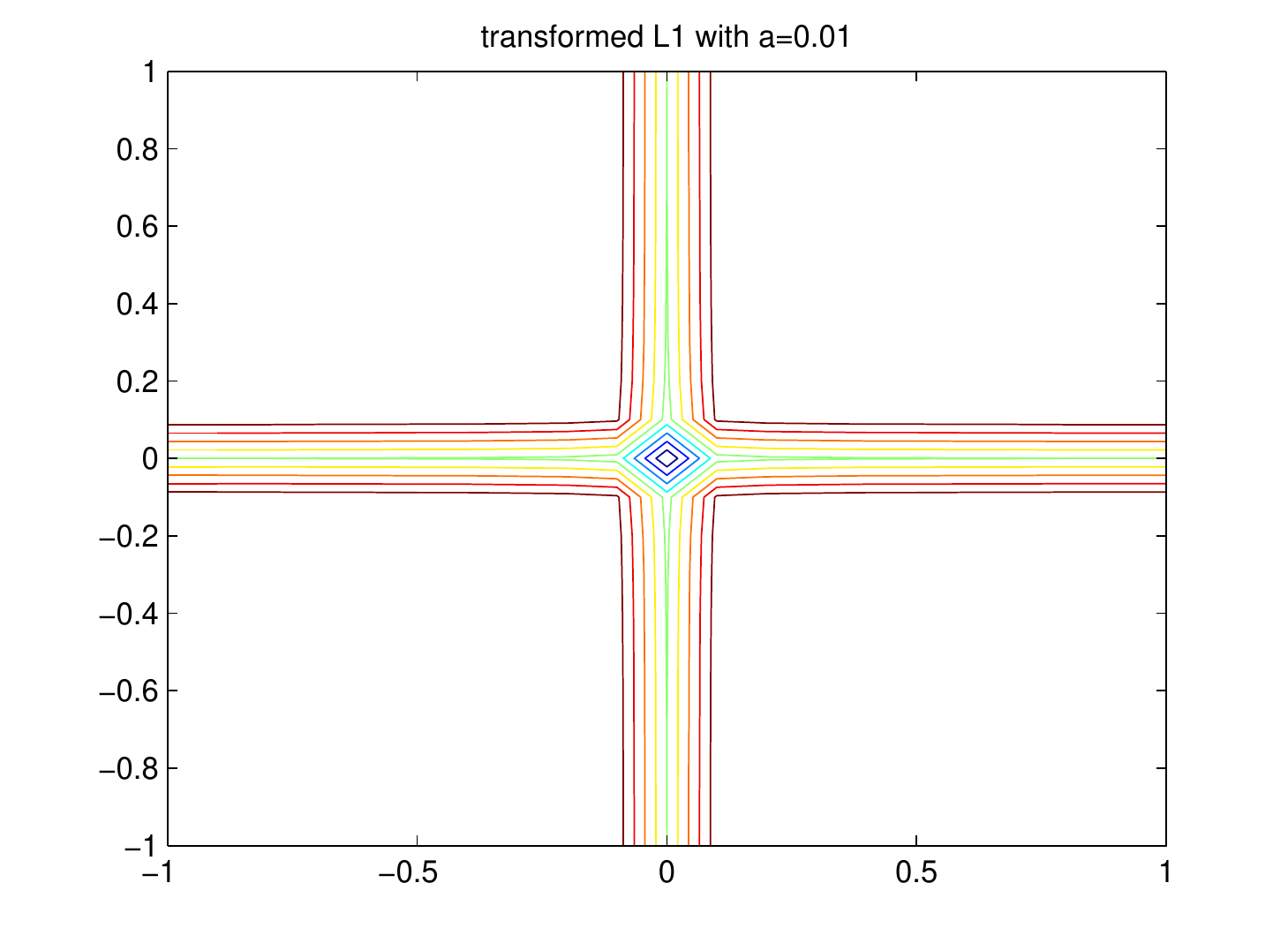}\quad
\includegraphics[width=0.48\textwidth,height=0.23\textheight]{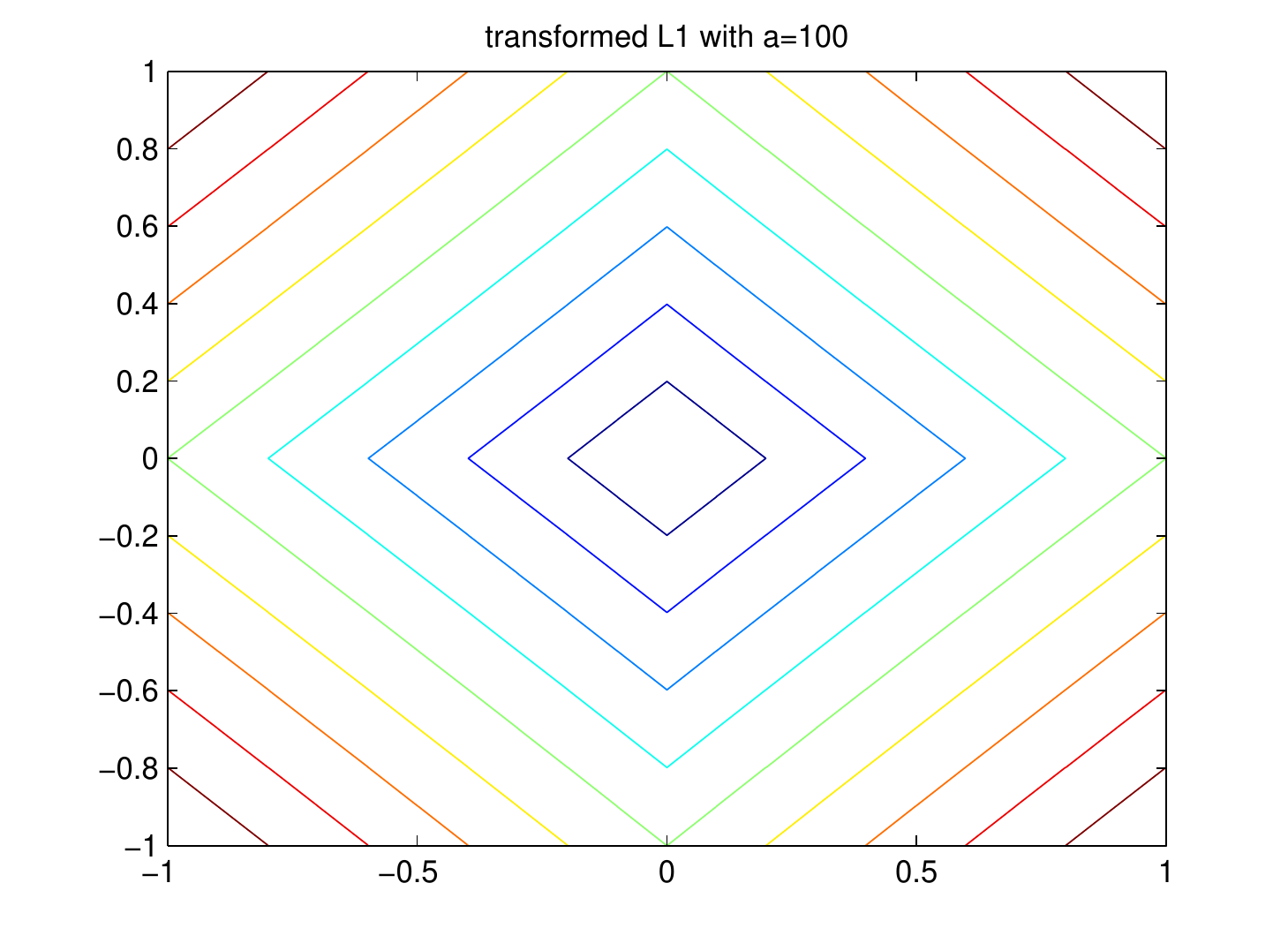}
\caption{Level lines of two sparsity promoting measures. Compared with $\ell_1$, the level lines of TL1 is closer to the axes or those of $\ell_0$ when parameter $a$ is small (e.g., $a=0.01$). When $a$ becomes larger (e.g., $a=100$), the level lines converge to those of $\ell_1$.}\label{fig:contours}
\end{figure}

The recoverability of the general TL1 minimization problem
\begin{align}\label{eq:TL1}
\min_{x\in\mathbb{R}^N}P_a(x) \quad \text{subject to}\ Bx=b,
\end{align}
with $P_a$ being some penalty function defined in \eqref{eq:P_a},
has been studied in \cite{Zhang_2016DCATL1} in the context of exact recovery, which assume the true solution is $s$-sparse. The major results from \cite{Zhang_2016DCATL1} are summarized as follows.
\begin{Theorem}[\cite{Zhang_2016DCATL1}]\label{thm:oldthm}
For a given sensing matrix $B$, if there is a number $R>|T|$, such that
\begin{align}\label{eq:penaltycon1}
\delta_R+\frac{R}{|T|}\frac{a^2}{(a+1)^2}\delta_{R+|T|}<\frac{R}{|T|}\frac{a^2}{(a+1)^2}-1,
\end{align}
then there exists $a^{\ast}>0$, depending only on matrix $B$, such that for any $a>a^{\ast}$, the minimizer $\hat{x}$ for \eqref{eq:gPC_TL1} is unique and equal to the minimizer $\tilde{x}_0$ in \eqref{eq:gPC_L0}.
\end{Theorem}
\begin{Remark}
If we choose $R=3|T|$ and $a$ goes to $+\infty$, the condition \eqref{eq:penaltycon1} becomes:
\begin{align*}
\delta_{3|T|}+3\delta_{4|T|}<2.
\end{align*}
which is exactly the condition $(1.6)$ of Theorem 1.1 in \cite{Candes_2005error}. This is consistent with the fact that when the parameter $a$ goes to $+\infty$, the penalty function $P_a(x)$ recovers the $\ell_1$ norm.
\end{Remark}

In the rest of this paper, we focus on the main result by applying the TL1 minimization to recover the sparse gPC coefficients in \eqref{eq:gPC_TL1}.

\section{Recovery gPC coefficients via TL1 minimization}\label{sec:TL1_pro}
    In this section, we establish the recovery results for the TL1 minimization problem \eqref{eq:gPC_TL1}. First of all we present the recoveriability for the general TL1 minimization problem \eqref{eq:TL1}. For the noiseless recovery, we improve the estimates in \cite{Zhang_2016DCATL1} and we also set up the noisy recovery.

\begin{Theorem}[Noiseless recovery]\label{thm:noiselessrecover}
Assume that $\delta_{2s}<\frac{1}{1+\frac{a+1}{a}\sqrt2}$, then the solution $\hat{x}$ to \eqref{eq:TL1} satisfies
\begin{equation*}
\|\hat{x}-x\|_2\leq C_0s^{-1/2}P_a(x-x_s),
\end{equation*}
where $C_0$ depends on the isometry constant of $B$ and is defined as
\begin{equation*}
C_0=\frac{(6\sqrt2a-2a+2\sqrt2)\delta_{2s}+2a}{a-((\sqrt2+1)a+\sqrt2)\delta_{2s}}.
\end{equation*}
In particular, if $x$ is $s$-sparse, the recovery is exact.
\end{Theorem}

Now, we consider the more applicable situation where the measurements are contaminated by noise. We assume
\begin{equation*}
b=Bx+e,
\end{equation*}
where $e$ is a perturbation vector with $\|e\|_2\leq\varepsilon$ for some $\varepsilon\geq0$. Under such circumstance, we propose the corresponding "\emph{de-noising}" TL1 recovery version:
\begin{align}\label{eq:PaTL1mindenoise}
\min_{x\in\mathbb{R}^N} P_a(x) \quad \text{subject to}\ \|Bx-b\|_2\leq\varepsilon.
\end{align}

In the next theorem, we indicate that stable reconstruction can be obtained by TL1 recovery under the same hypotheses as in Theorem \ref{thm:noiselessrecover}.
\begin{Theorem}[Noisy recovery]\label{thm:noisyrecover}
Assume that $\delta_{2s}<\frac{1}{1+\frac{a+1}{a}\sqrt2}$ and $\|e\|_2\leq\varepsilon$, then the solution $\hat{x}$ to \eqref{eq:PaTL1mindenoise} satisfies
\begin{equation}\label{eq:TL1noisybound}
\|\hat{x}-x\|_2\leq C_0s^{-1/2}P_a(x-x_s)+C_1\varepsilon,
\end{equation}
with $C_0$ the same as in Theorem \ref{thm:noiselessrecover} and $C_1$ given explicitly as
\begin{align*}
C_1=\frac{2(2a+1)\sqrt{1+\delta_{2s}}}{a-((\sqrt2+1)a+\sqrt2)\delta_{2s}}.
\end{align*}
\end{Theorem}
The proof can be found in Appendix \ref{apx:prof_noisyrecover}. Theorem \ref{thm:noiselessrecover} can be proved by the first part of \eqref{eq:TL1noisybound}. Noting that, from now on, $x_T$ refers to the vector that equals to $x$ on an index set $T$ and zero elsewhere.

\begin{Remark}
The results here are more elegant than that in \cite{Zhang_2016DCATL1} recalled in Theorem \ref{thm:oldthm}. In \cite{Yan_2015L12}, the authors studied the $\ell_{1-2}$ minimization and they discussed the properties of sparse approximation when using $\ell_{1-2}$ minimization. In their work, they derived an estimate of $\delta_{2s}<\frac{1}{1+\sqrt2+\frac{2\sqrt2}{\sqrt{s}-1}}$ for non-sparse solutions. While, in this paper, when we use the TL1 minimization as the CS solver, we derive that $\delta_{2s}<\frac{1}{1+\frac{a+1}{a}\sqrt2}$, hence if $a>\frac{\sqrt{s}-1}{2}$, then the recoverability of TL1 is better than that of $\ell_{1-2}$.
\end{Remark}

We obtain the following important properties of the recovery analysis of gPC coefficients via TL1 minimization.
\subsection{Recovery properties}
According to \cite{Rauhut_2012sparseLegen}, we first recall the definition of a bounded orthonormal system, if a orthogonal polynomial system is uniformly bounded,
\begin{align}\label{eq:bounnedsystem}
\sup_{n}\|\Phi_n\|_{\infty}=\sup_n\sup_{z}|\Phi_n(z)|\leq K \quad (K\geq1).
\end{align}

\begin{Theorem}[RIP for bounded orthonormal systems \cite{Rauhut_2012sparseLegen}]\label{thm:boundedBOS}
Let $A\in\mathbb{R}^{M\times N}$ be the interpolation matrix with entries $\{a_{j,n}=\Phi_n(z^j)\}_{1\leq n\leq N,1\leq j\leq M}$ from \eqref{eq:matrixelem}, where $\{\Phi_n\}$ is a bounded orthonormal system satisfying \eqref{eq:bounnedsystem}. Assume that
\begin{equation*}
M\geq C\delta^{-2}K^2s\log^3(s)\log(N),
\end{equation*}
then with probability at least $1-N^{-\gamma\log^3(s)}$, the RIC $\delta_s$ of $1/\sqrt{M}A$ satisfies $\delta_s\leq\delta$. Here the $C,\gamma>0$ are universal constants.
\end{Theorem}

Based on the above theorem and combine the previous recovery properties of TL1 minimization, we can obtain the following key theorem that plays an essential role in reconstructing gPC Legendre approximation to the unknown function. This part generally follows the lines of \cite{Yan_2012Sc, GuoLiu_2017Lq} and we define the basis set $T_P^d$ to be orthonormal polynomials up to degree $P$, i.e., $T_P^d = \{\Phi_\alpha : |\alpha|\leq P\}$.

\begin{Theorem}[Recovery of  polynomial functions]\label{thm:poly_recover}
Let $A$ be the Vandermonde-like measurement matrix established using the Legendre polynomials in $T_{P}^d$ and $M$ sampling points $\{\xi_i\}_{i=1}^M$ from the uniform measure. Suppose $d\geq P$ and
\begin{align}
M\geq C\delta^{-2}3^Ps\log^3(2s)\log(N),
\end{align}
where $\delta<\frac{1}{1+\frac{a+1}{a}\sqrt2}$ and $C>0$ is universal. Let the target function be a Legendre series as \eqref{eq:truncatePCE}, i.e.,
\begin{align*}
f(z)=\sum_{n=1}^N{x}_n\Phi_n(z),
\end{align*}
with an arbitrary coefficient vector $x\in\mathbb{R}^N$, $N=\dim T_P^d$, and $x_s$ be its truncated vector corresponding to the $s$ largest entries (in absolute value). Then, with a probability exceeding $1-N^{-\gamma\log^3(2s)}$, the following results hold.
\begin{itemize}
  \item Let $b=Ax$, then the solution $\hat{x}$ to \eqref{eq:gPC_TL1} satisfies
  \begin{align*}
  \|\hat{x}-x\|_2\leq C_0s^{-1/2}P_a(x-x_s),
  \end{align*}
  \item Let $b=Ax+e$, where $e\in\mathbb{R}^M$ is any perturbation with $\|e\|_2\leq\varepsilon$, then the
solution $\hat{x}$ to \eqref{eq:PaTL1mindenoise} satisfies
\begin{equation}
\|\hat{x}-x\|_2\leq C_0s^{-1/2}P_a(x-x_s)+C_1\varepsilon.
\end{equation}
where ${C}_0$ and ${C}_1$ are defined in  \eqref{eq:TL1normalC}.
\end{itemize}
\end{Theorem}
\begin{proof} Since $M\geq C\delta^{-2}3^Ps\log^3(s)\log(N)$, we can obtain from Theorem \ref{thm:boundedBOS} directly that the restricted isometry constant $\delta_s$ of $\frac{1}{\sqrt{M}}A$ satisfies $\delta\leq\delta_{2s}\le \frac{1}{1+\frac{a+1}{a}\sqrt2}$. The proof then follows immediately from Theorems \ref{thm:noiselessrecover} and \ref{thm:noisyrecover}.
\end{proof}
\begin{Remark}
Essentially, the errors in the Theorem \ref{thm:poly_recover} can be bounded by $Cs^{-1/2}\|x-x_s\|_1$ with some constant $C$.
\end{Remark}

\subsection{Recover algorithm for constrained TL1 minimization}
In this section, we propose the algorithm to solve the constrained TL1 minimization problem using the DCA framework, namely,
\begin{align*}
\begin{split}
&\min_{x\in\mathbb{R}^N} P_a(x) \ \text{subject to}\ Ax=b, \\
&\Leftrightarrow \\
&\min_{x\in\mathbb{R}^N}\frac{a+1}{a}\|x\|_1-\{\frac{a+1}{a}\|x\|_1-P_a(x)\} \ \text{subject to}\ Ax=b.
\end{split}
\end{align*}
First, we recall the original DCA-TL1 algorithm in Algorithm \ref{alg_1} which has been introduced in \cite{Zhang_2016DCATL1} and here $shrink(\cdot,\cdot)$ is given by:
\begin{align*}
shrink(x,r)_i \triangleq sgn(x_i)\max\{|x_i|-r,0\},
\end{align*}
where $sgn(t)\triangleq \left\{\begin{array}{ll}1, & t>0;\\ 0, &t=0;\\ -1, &t<0.\end{array}\right.$

\begin{algorithm}\label{alg_1}
\SetAlgoNoLine
\caption{DCA-TL1 (pseudo-code)}
Prescribe $\epsilon_{outer}>0, \epsilon_{inner}>0$, $\delta \gg 1$ and set $x^0=0, x^1\neq0$, $n=1$ \\
\While {$\|x^n-x^{n-1}\|>\epsilon_{outer}$}{
$z=(z_1, \ldots, z_N)$, where $z_i=\frac{a+1}{a}sgn(x^n_i)-\frac{(a+1)sgn(x^n_i)}{a+|x^n_i|}+\frac{(a+1)x^n_i}{(a+|x^n_i|)^2}$
$x_0=0, x_1=x^n, i=1, y_i=x_i, u_i=v_i=0$ \\
\While{$\|x_{i}-x_{i-1}\|>\epsilon_{inner}$}{$x_{i+1}=(A^{\top}A+I)^{-1}(A^{\top}b+y_i+\frac{z-u_i-A^{\top}v_i}{\delta})$\\
$y_{i+1}=shrink(x_{i+1}+\frac{u_i}{\delta}, \frac{a+1}{a\delta})$\\
$u_{i+1}=u_i+\delta(x_{i+1}-y_{i+1})$\\
$v_{i+1}=v_i+\delta(Ax_{i+1}-b)$}
$n=n+1$\\
$x^n=x_i$
}
$x = x^n$
\end{algorithm}

It is worth noting that in the above DCA-TL1 algorithm, the parameter $a$ is also very important. As shown in Figure \ref{fig:contours}, when $a$ tends to zero, the penalty function approaches the $\ell_0$ norm; If $a$ goes to $\infty$, the objective function will be more convex and act like the $\ell_1$ norm. Therefore, choosing an appropriate $a$ will improve the effectiveness and success rate of our algorithm. In this paper, we select parameter $a$ through testing DCA-TL1 on recovering sparse vectors with different values of $a$, varying among $\{0.1\ 0.2\ 0.3\ 1\ 2\ 10\}$. According to the work in \cite{Zhang_2016DCATL1}, we do the similar tests using the bounded orthonormal systems to choose the optimal parameter $a$. Figure \ref{fig:varyinga} shows the computed results with different values of $a$, we see that DCA-TL1 with $a=0.3$ is the best among all tested values for recovering the sparse functions. Thus for the following tests in recovering sparse functions, we set parameter $a=0.3$.

\begin{figure}[!ht]
\centering
\includegraphics[width=0.48\textwidth,height=0.23\textheight]{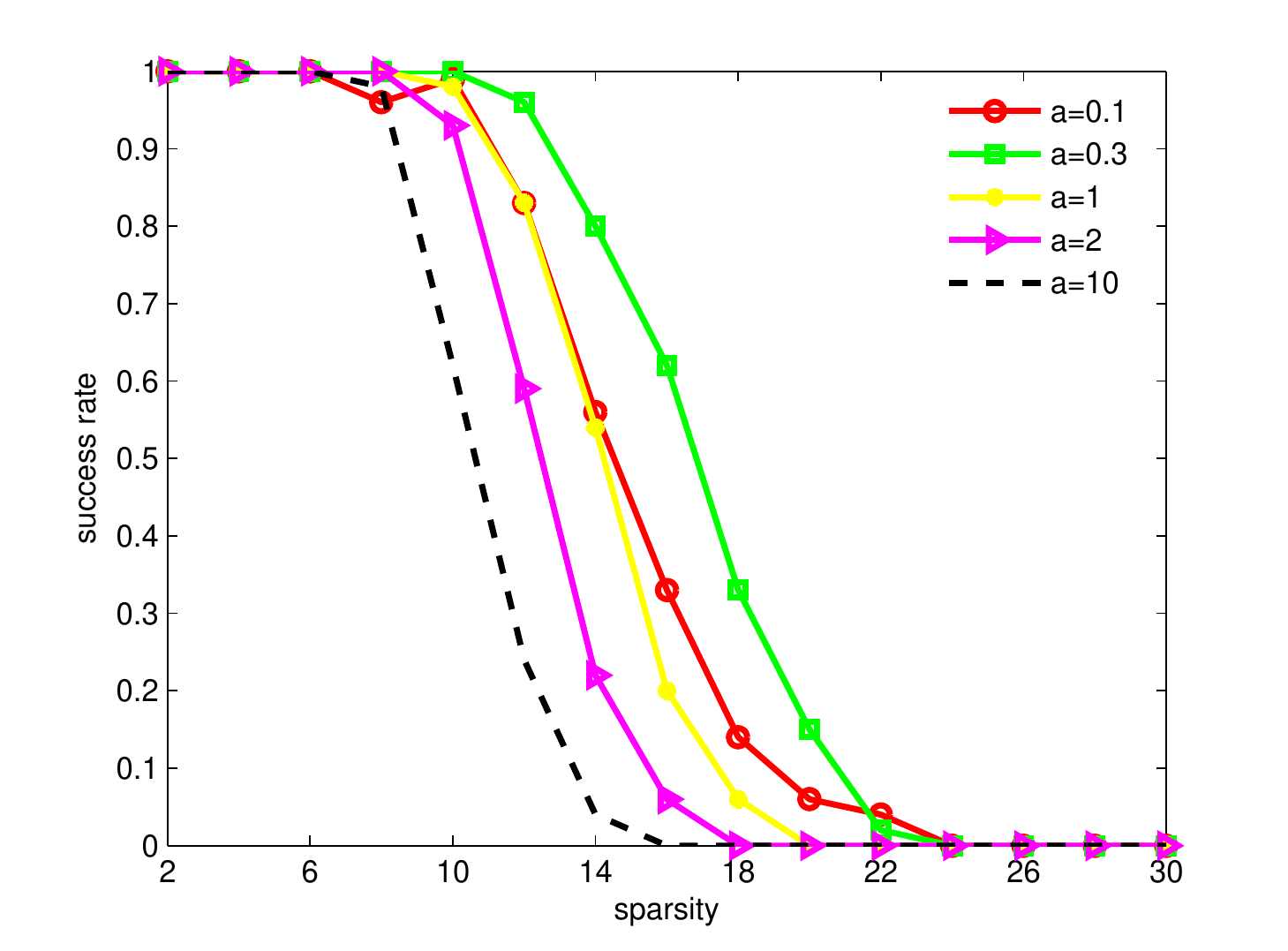}\quad
\includegraphics[width=0.48\textwidth,height=0.23\textheight]{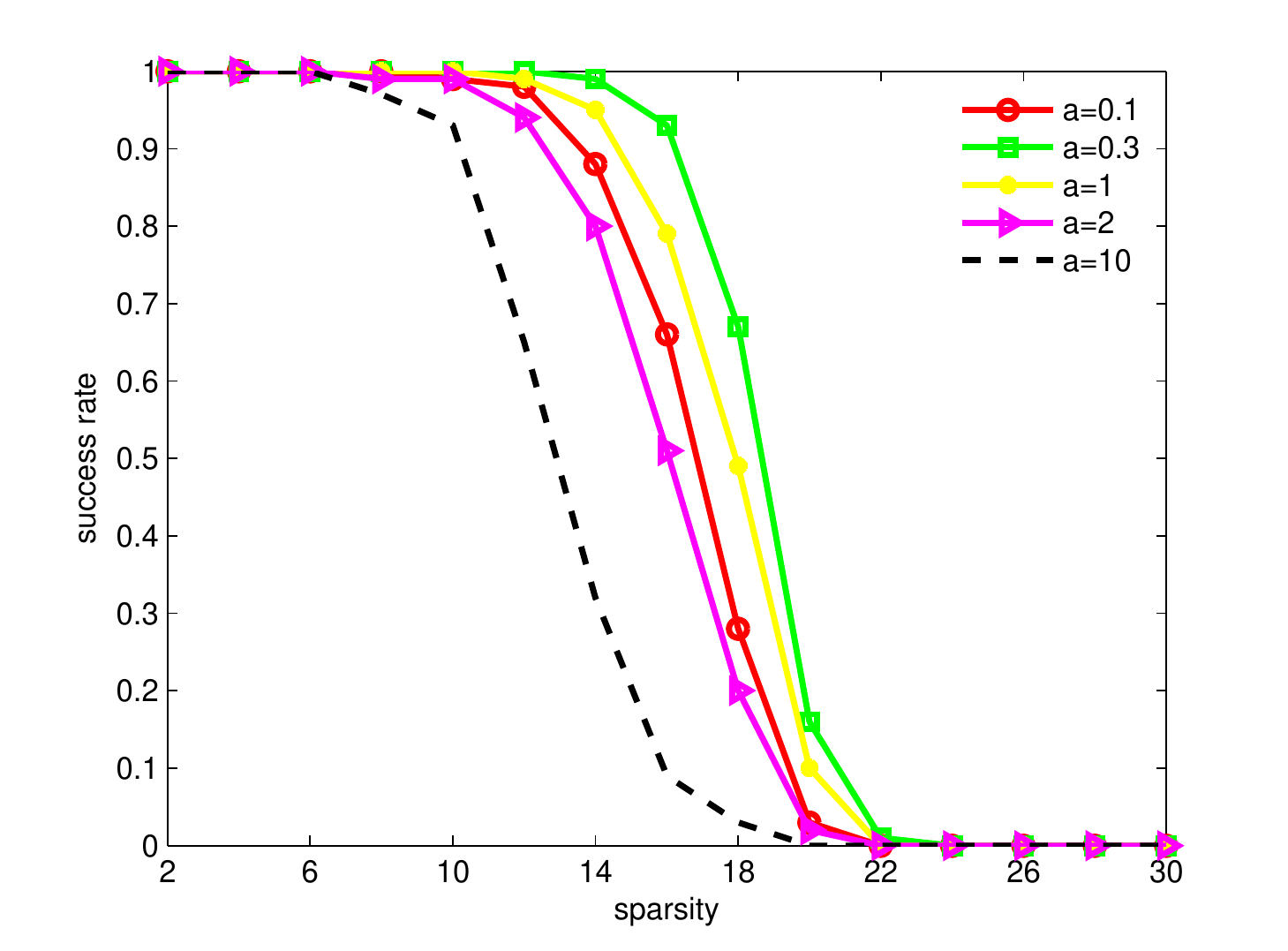}
\caption{Success  rate against sparsity $s$ for varying parameter $a$ using Legendre polynomials. Left: $d=2$, $k=20$, $N=231$, $M=45$; Right: $d=10$, $k=4$, $N=1001$, $M=70$}\label{fig:varyinga}
\end{figure}

If the target function is non-sparse, due to the various types of these functions, we can not easily prescribe an appropriate value of $a$. Hence, we develop an adaptive DCA-TL1 minimization algorithm, that is, prescribing a set of candidate values of $a$ solving the TL1 minimization for each candidate and selecting the sparsest solution as the solution to problem \eqref{eq:gPC_TL1}. The theoretical analysis in Section \ref{sec:TL1_pro} indicates that if $a$ tends to zero we gets sparser solutions, and the numerical tests in Figure \ref{fig:varyinga} also show that. Hence, in the later numerical examples, we choose $c = [0.2, 0.3, 1]$ as the candidate set of $a$ for low dimensional cases and $c = [0.05, 0.1, 0.2, 0.3, 1]$ as the candidate set of $a$ for more complicate high dimensional cases. We perform the algorithm 1 for each value of $a$ from the candidate set and retain the sparest solution $x$. The pseudo-code is given in Algorithm \ref{alg_2}.

\begin{algorithm}\label{alg_2}
\SetAlgoNoLine
\caption{Adaptive DCA-TL1 for non-sparse recovery}
Define $Sparsity=N$, given a vector $c$, $L=\text{length}(c)$ \\
\For{l=1:L}{
$a_l=c(l)$;\\
$xt=\bar{x}$, where $\bar{x}$ is obtained from Algorithm \ref{alg_1} with a fixed parameter $a = a_l$;\\
$s$=length(find($|xt|>1e-6$));\\
\If {$s<Sparsity$}{
 $Sparsity=s$;\\
 $x=xt$;\\
}
}
\end{algorithm}

\section{Numerical examples}
In this section, we present some numerical examples to demonstrate the theoretical findings of the TL1 minimization for uncertainty quantification. We proceed the numerical tests in two ways. On the one hand, assuming that the target functions are sparse polynomial expansions with arbitrarily chosen sparse coefficients we focus on the successful recovery probability. On the other hand, given that the target functions are non-sparse in general we examine the approximation errors. Meanwhile, to demonstrate the efficiency of TL1 algorithm for high-dimensional polynomial approximation, we conduct two experiments in each example, one for low dimensional problem and the other for high dimensional problem. In all implementations, $A$ is a bounded orthonormal system with each component $a_{ij}$ generated by uniform distributed samples on $[-1, 1]^d$ and the corresponding Legendre polynomials. Furthermore, for each case we repeat 100 times to obtain an average result to reduce the statistical oscillations.

We employ the algorithm \ref{alg_2} to solve the TL1 minimization problem \eqref{eq:gPC_TL1}. In both cases, we also perform the DCA-$\ell_{1-2}$ minimization established in \cite{Lou_2015ComputeL12} and the standard $\ell_1$ minimization using SPGL1 package \cite{Berg_2007spgl} for comparison. Other optimization algorithms can be employed for TL1-minimization, c.f. \cite{ZhangXin2017a,ZhangXin2017b}.

\subsection{Exact recovery for sparse functions}
In this example, we choose a sparsity level $s$ and $x^{\ast}$ to be a randomly generated sparse vector with sparsity $s$, i.e., the components of $x^{\ast}$ have $s$ non-zero values drawn from a standard (namely, zero mean and unit variance) Gaussian distribution, while the rest components are zero. This $s$-sparse vector $x^{\ast}$ is set to be the coefficients of the target polynomial expansion, i.e.,
\begin{equation*}
f(y) = \sum_{i=1}^N x^{\ast}_i\Phi_i(y),
\end{equation*}
and we seek to recover the target coefficients via TL1 minimization. In the following experiments, a recovery is considered successful if the resulting coefficient vector $x$ satisfies $\|x-x^{\ast}\|_{\infty}<10^{-3}$.

\subsubsection{Low-dimensional case when $d=2$} We first examine the performance of the methods in a relatively low dimension of $d=2$, the polynomial space $T_k^d$ is fixed at order $k=20$ and the number of total unknowns $N=\dim T_k^d=231$.

\begin{figure}[!ht]
\centering
\includegraphics[width=0.48\textwidth,height=0.23\textheight]{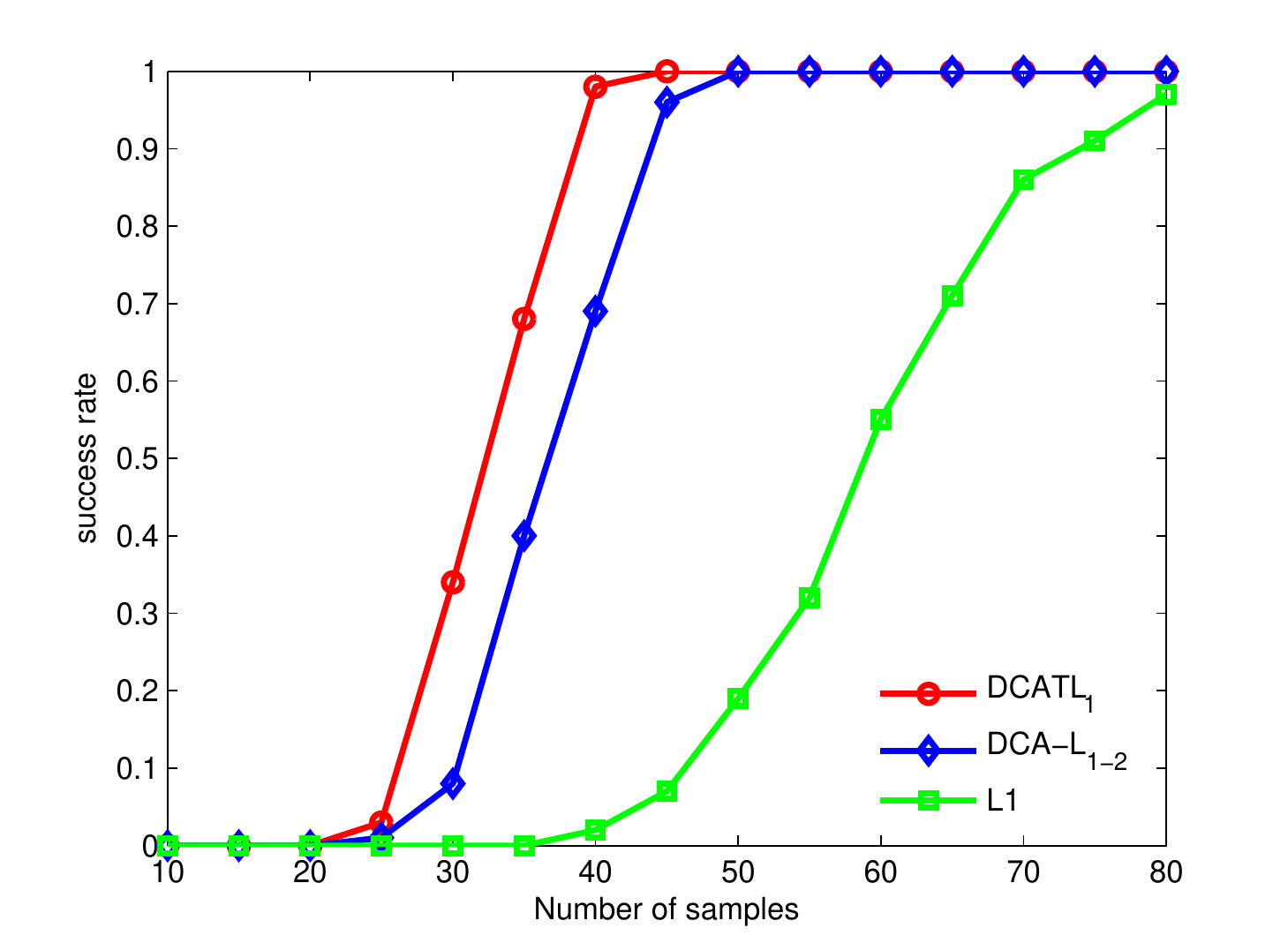}\quad
\includegraphics[width=0.48\textwidth,height=0.23\textheight]{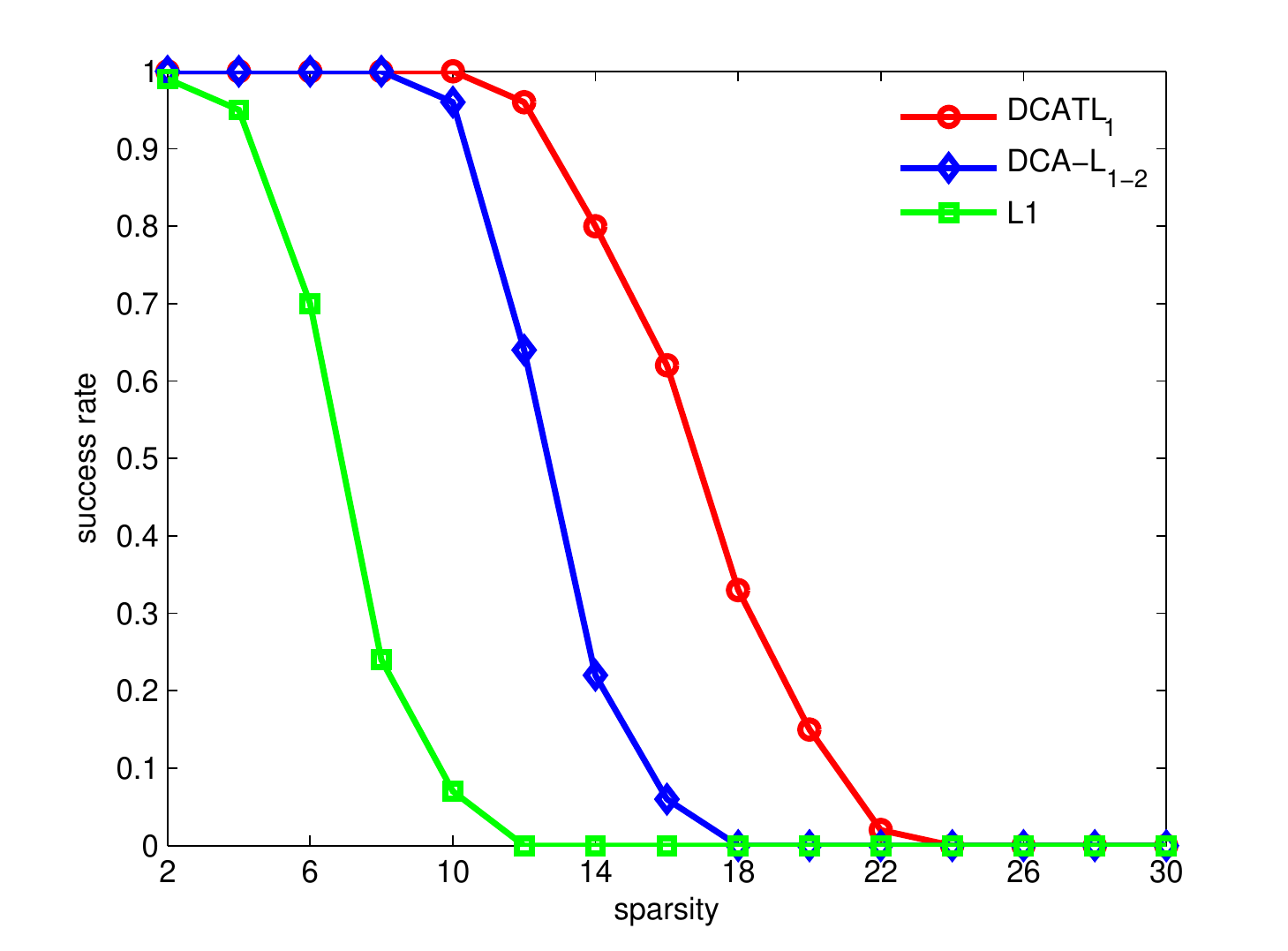}
\caption{Recovery results in $d=2$, $k=20$, $N=231$. Left: Probability of successful recovery vs. number of samples with $s=10$; Right: Probability of successful recovery vs. sparsity with $M=45$.}\label{fig:lowdimtests}
\end{figure}

The probability of successful recovery with respect to the number of samples $(M)$ is displayed in Figure \ref{fig:lowdimtests}(left), for a fixed sparsity level $s=10$. We observe that the DAC-TL1 minimization performs the best and DCA-$\ell_{1-2}$ is better than the standard $\ell_1$ minimization. Figure \ref{fig:lowdimtests}(right) shows the success rate as a function of sparsity $s$. Similarly as before, we observe the best performance by the DCA-TL1 minimization, and the next by the DCA-$\ell_{1-2}$ minimization.

\subsubsection{High-dimensional examples when $d=10$}
We fix the polynomial space $T_k^d$ at $k=4$, which results in $N=\dim T_k^d=1001$. Results of these simulations are given in Figure \ref{fig:highdimtests}.

\begin{figure}[!ht]
\centering
\includegraphics[width=0.48\textwidth,height=0.23\textheight]{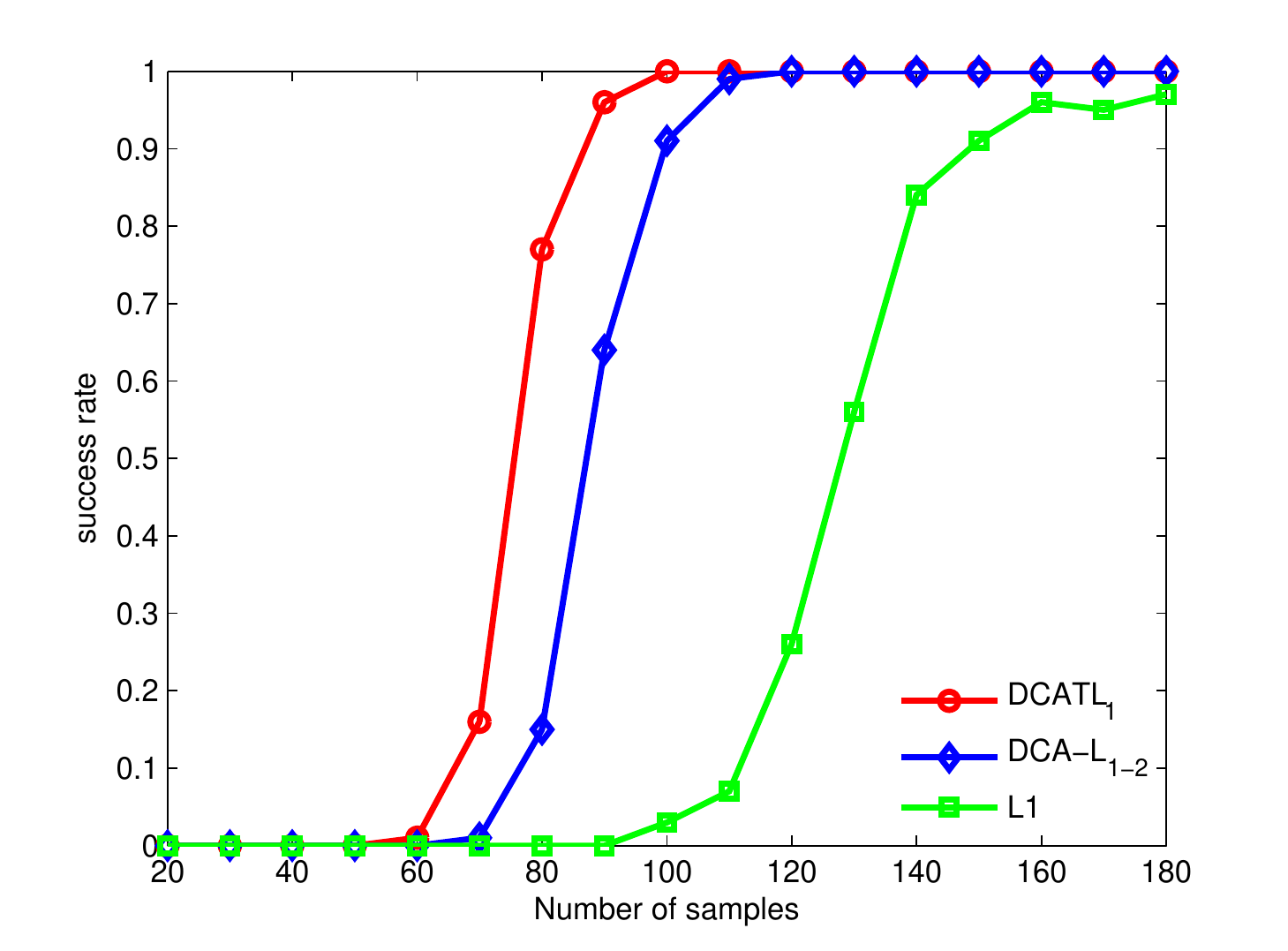}\quad
\includegraphics[width=0.48\textwidth,height=0.23\textheight]{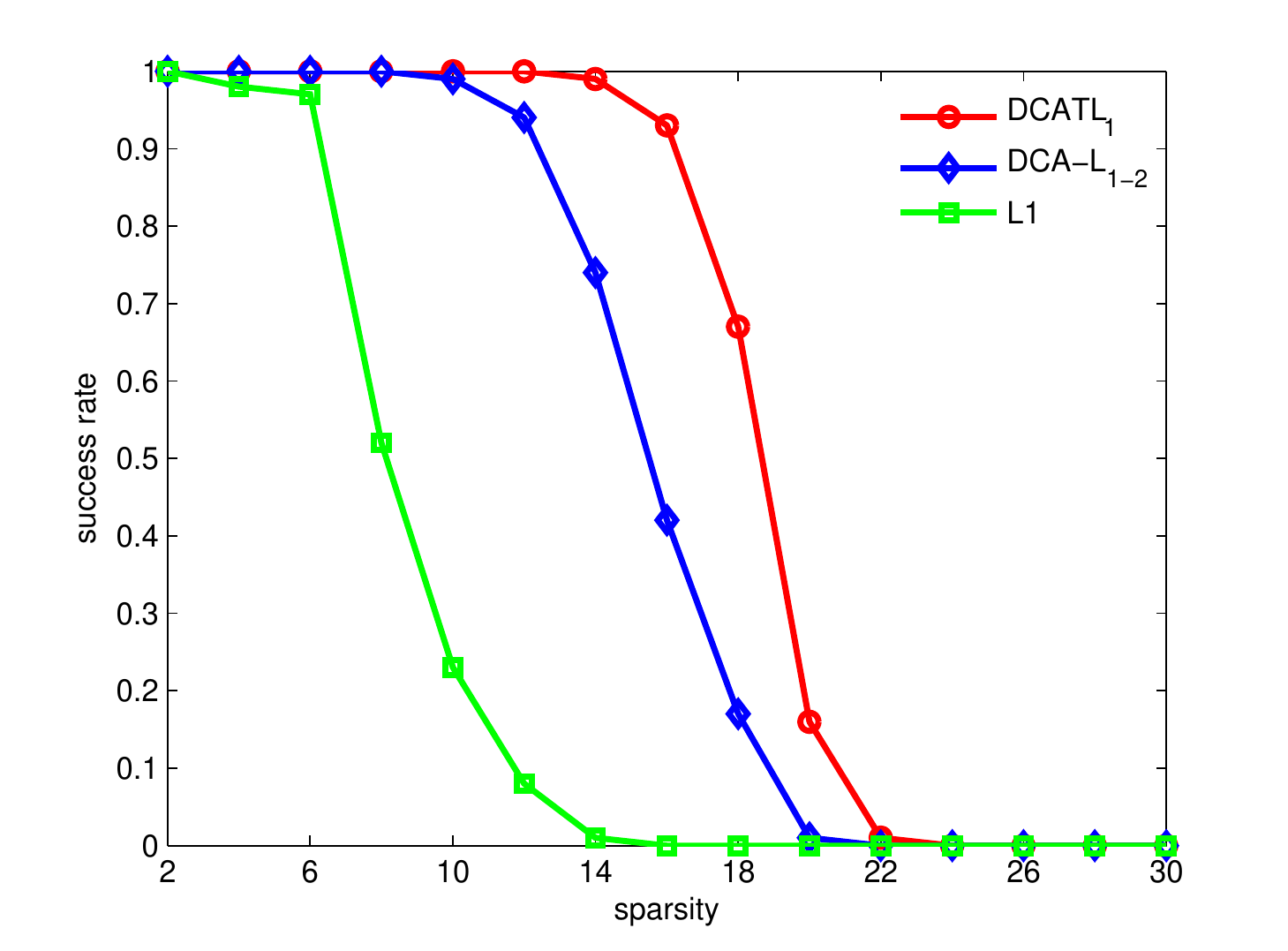}
\caption{Recovery results in $d=10$, $k=4$, $N=1001$. Left: Probability of successful recovery vs. number of samples with $s=20$; Right: Probability of successful recovery vs. sparsity with $M=70$.}\label{fig:highdimtests}
\end{figure}

In Figure \ref{fig:highdimtests}, the plot on the left shows the recovery probabilities with respect to the number of sample points at a fixed sparsity level $s=20$, while the plot on the right provides recovery rate versus the sparsity level with a fixed number of samples $(M=70)$. Again, for this higher dimensional test, the DCA-TL1 minimization outperforms the other two algorithms and DCA-$\ell_{1-2}$ performs better than standard $\ell_1$ minimization.

\subsection{Analytical function approximations}
In this part we implement the numerical tests for approximating general non-sparse functions. We use the discrete $\ell_2$ error to measure the performance. To be specific, given a set of $Q=2000$ random samples $\{z^i\}_{i=1}^Q$ drawn from the density $w$ we evaluate the true function $f(z^i)$ and the gPC approximation $\tilde{f}(z^i)$ obtained from solving the TL1 minimization and compute the error

\begin{figure}[!ht]
\centering
\includegraphics[width=0.48\textwidth,height=0.23\textheight]{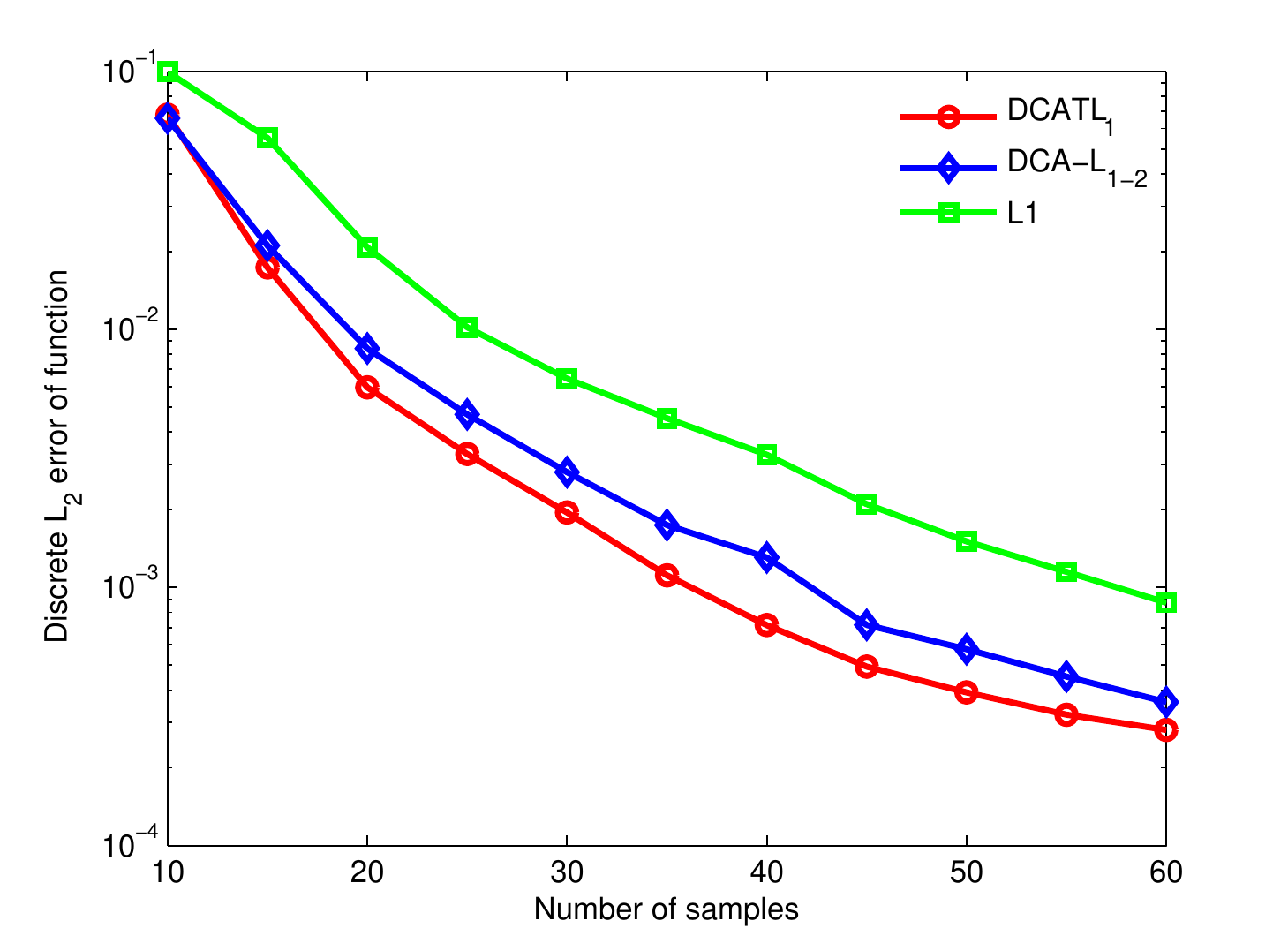}\quad
\includegraphics[width=0.48\textwidth,height=0.23\textheight]{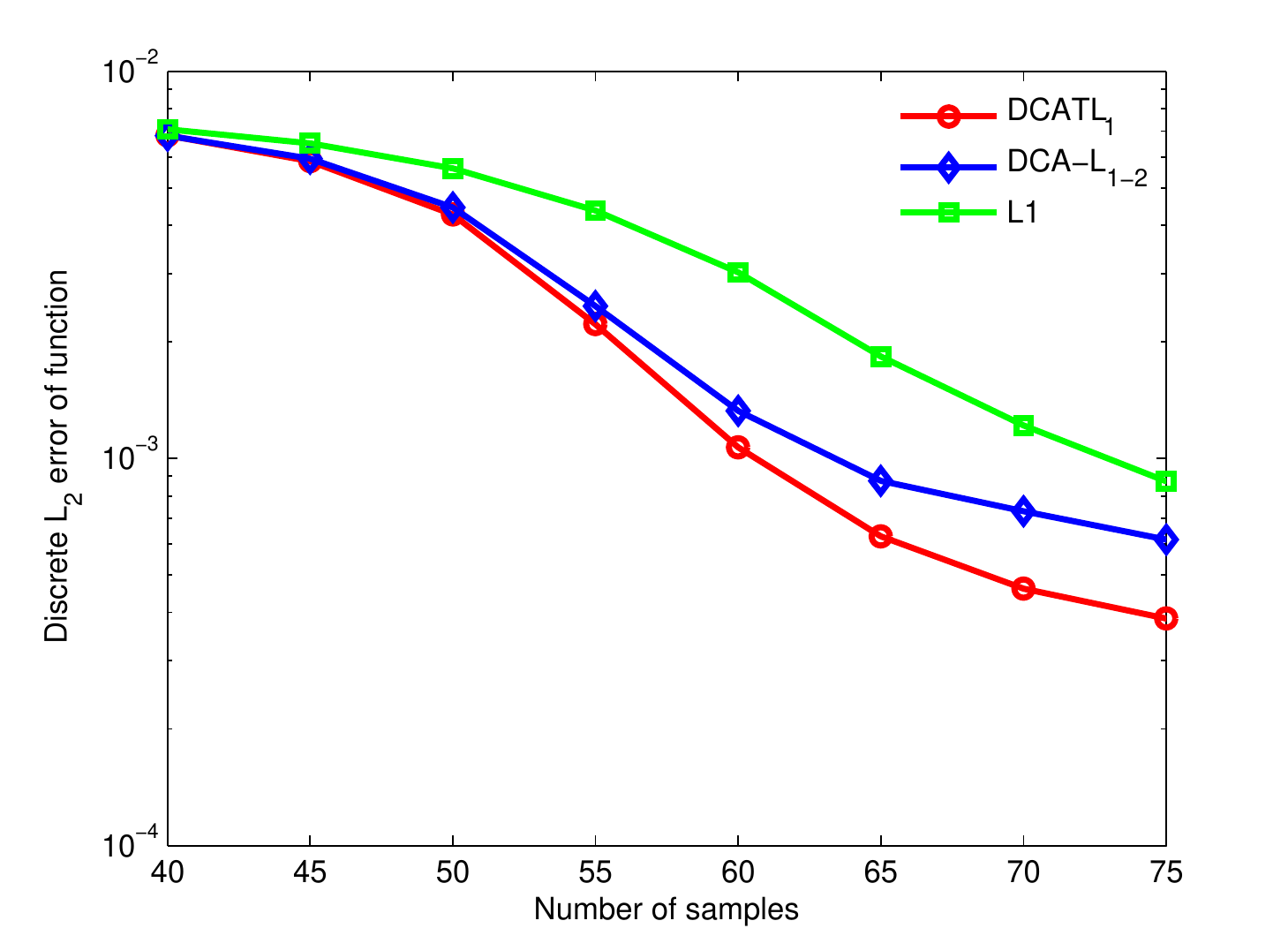}
\caption{Discrete $\ell_2$ error versus number of sample points for $f_1(z)$. Left: low-dimension, $d=2, k=20, N=231$; Right: high-dimension, $d=10, k=4, N=1001$.}\label{fig:f1}
\end{figure}

We consider the following multi-dimensional analytical functions and build their gPC approximations:
\begin{align*}
\begin{split}
&f_1(z)=\frac{1}{\sum_{i=1}^d0.5+0.1z_i},\\
&f_2(z)=\Big(1+\frac{1}{2d}\sum_{i=1}^d\frac{i-1/2}{d}(z_i+1)\Big)^{-d-1}.
\end{split}
\end{align*}

\begin{figure}[!ht]
\centering
\includegraphics[width=0.48\textwidth,height=0.23\textheight]{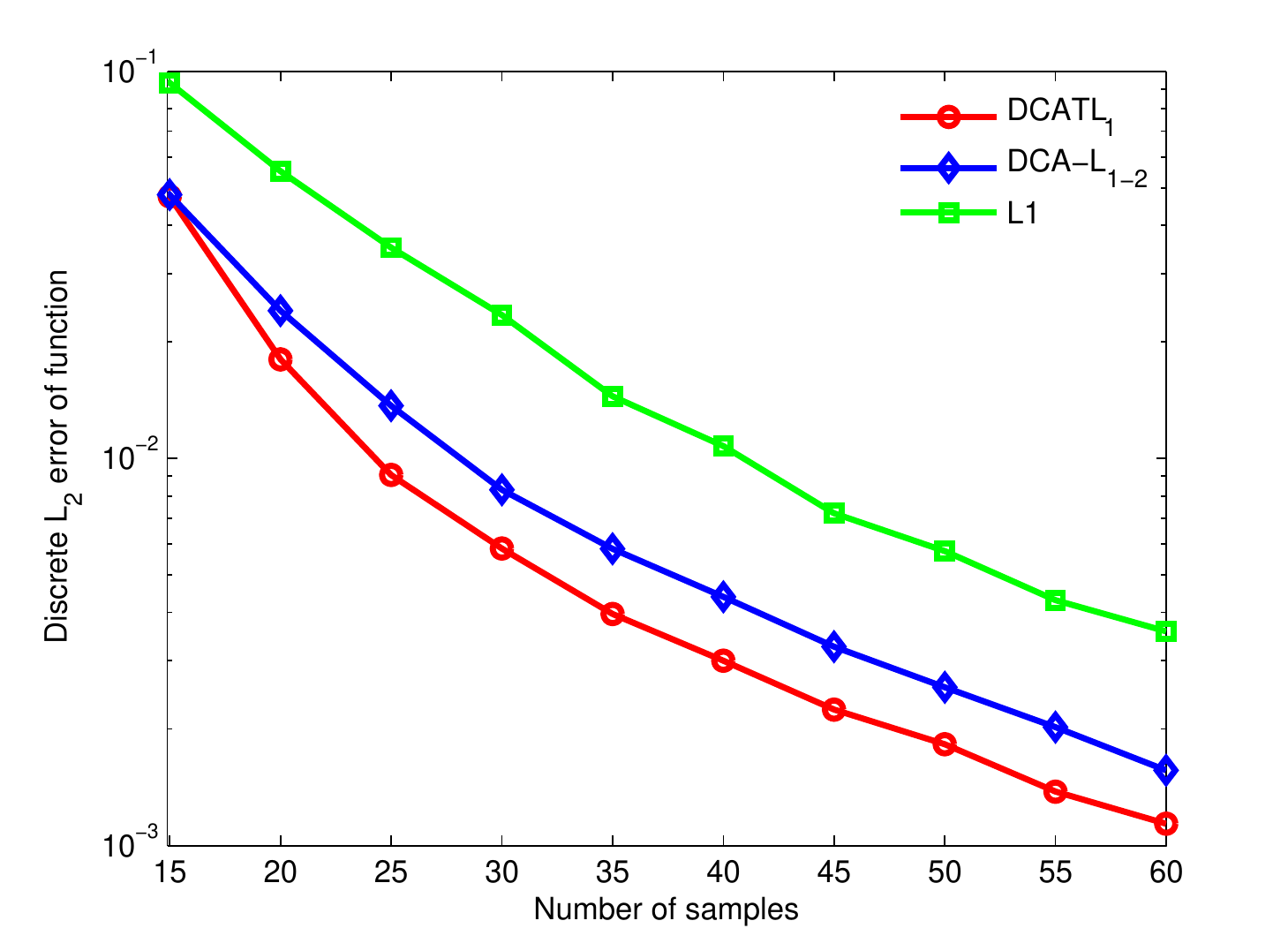}\quad
\includegraphics[width=0.48\textwidth,height=0.23\textheight]{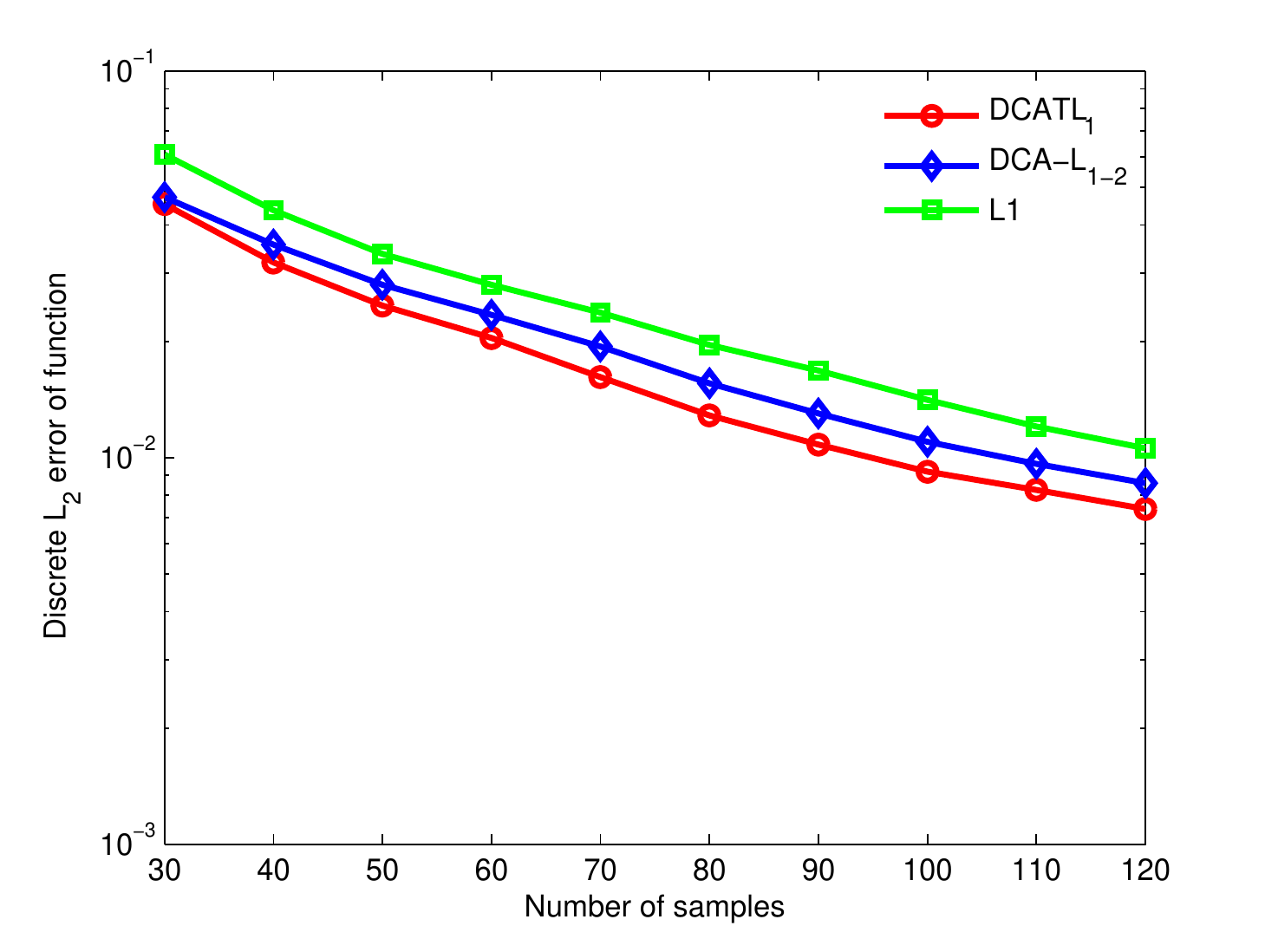}
\caption{Discrete $\ell_2$ error versus number of sample points for $f_2(z)$. Left: low-dimension, $d=2, k=20, N=231$; Right: high-dimension, $d=6, k=5, N=486$.}\label{fig:f2}
\end{figure}

Numerical results are displayed in Figures \ref{fig:f1} and \ref{fig:f2}, illustrating the superiority of DCA-TL1 minimization. To be specific, on the left hand side of Figure \ref{fig:f1}, we test the approximation error of $f_1(z)$ in low dimensional case, and on the right hand side of Figure \ref{fig:f1} we consider the high dimensional case. It is obviously that DCA-TL1 algorithm again gets the best performance and much better than the other two methods in both low dimensional ($d=2$) case and high dimensional ($d=6$) case. Numerical results for $f_2(z)$ are given in Figure \ref{fig:f2}. From this figure, we can see that DCA-TL1 minimization gains much better approximation than the other two minimization algorithms in the low dimensional ($d=2$) case and slightly better approximation in the high dimensional ($d=6$) case.

\subsection{KdV equation with random force}
Finally, we consider the solution to the Korteweg de Vries (KdV) equation with random force \cite{Zheng_2015MEPCM}:
\begin{align}\label{eq:Kdv}
u_t+2uu_x+u_{xxx}=f(t;\xi),  \quad x\in[-70, 70], \ t\in[0,T].
\end{align}
with initial condition
\begin{align}\label{eq:Kdvini}
u(x,0)=\frac{3\nu}{2}\text{sech}^2(\frac{1}{2}\sqrt{a}(x-x_0)).
\end{align}
where $\nu$ is associated with the speed of the solution, $x_0$ is the initial position of the solution. In this experiment, the random force is defined by a linear combination of random fields, i.e.,
\begin{align*}
f(t;\xi)=\sigma\sum_{i=1}^d\sqrt{\lambda_i}\phi_i(t)\xi_i.
\end{align*}
where $\sigma$ is a constant and $\{\lambda_i, \phi_i(t)\}$ are the leading eigenpairs of the Karhunen-Lo\'eve (KL) expansion kernel \cite{Xiu_2010SCbook}:
\begin{align*}
C(x,x^{\prime})=\exp(\frac{|x-x^{\prime}|}{\ell_c}).
\end{align*}
In this application we set $\ell_c = 0.25, \sigma= 0.1$ and choose the quantity of interest (QoI) to be $u(x, t; \xi)$ at $x = -6.5878, t = 1$. We need to solve a deterministic problem of the KdV equation at each collocation point for the random input $\xi$.  Hence, in our numerical implementation, we use the Chebyshev-Gauss-Lobatto collocation discretization in physical space. For the time marching, we use the third-order Adams--Bashforth scheme for the nonlinear convection term $uu_x$ and then apply Crank--Nicolson scheme for the dispersion term $u_{xxx}$.

To examine the convergence of TL1 method for the above more complicated and nonlinear differential equation, we compare the relative averaged root mean square error (reRMSE) for the solution $u$ at $x=-6.5878$, $t=1$ between the reference solution and approximated solution computed by the standard $\ell_1$, $\ell_{1-2}$and TL1-minimization, using $100$ independent replications. Here
\begin{align}\label{def:er}
\textrm{reRMSE}=\frac{\Big(\frac{1}{Q}\sum_{i=1}^Q|\tilde{f}(z^i)-f(z^i)|^2\Big)^{1/2}}{\Big(\frac{1}{Q}\sum_{i=1}^Q|f(z^i)|^2\Big)^{1/2}}.
\end{align}
with $Q = 100$.
We test a low dimensional random input with $d=2, k=20$ and a high dimensional case with $d=10, k=4$. In Figure \ref{fig:kdverrmodel}, we plot the convergence behavior of $\ell_1$, $\ell_{1-2}$ and TL1 minimization for problem \eqref{eq:Kdv}, where the random variables are sampled from i.i.d. $U[-1,1]$ and the gPC-based polynomial is Legendre type. We see that with the increment of the total number of collocation points, the discrete $\ell_2$ error reduced. 

\begin{figure}[!ht]
\centering
\includegraphics[width=0.48\textwidth,height=0.23\textheight]{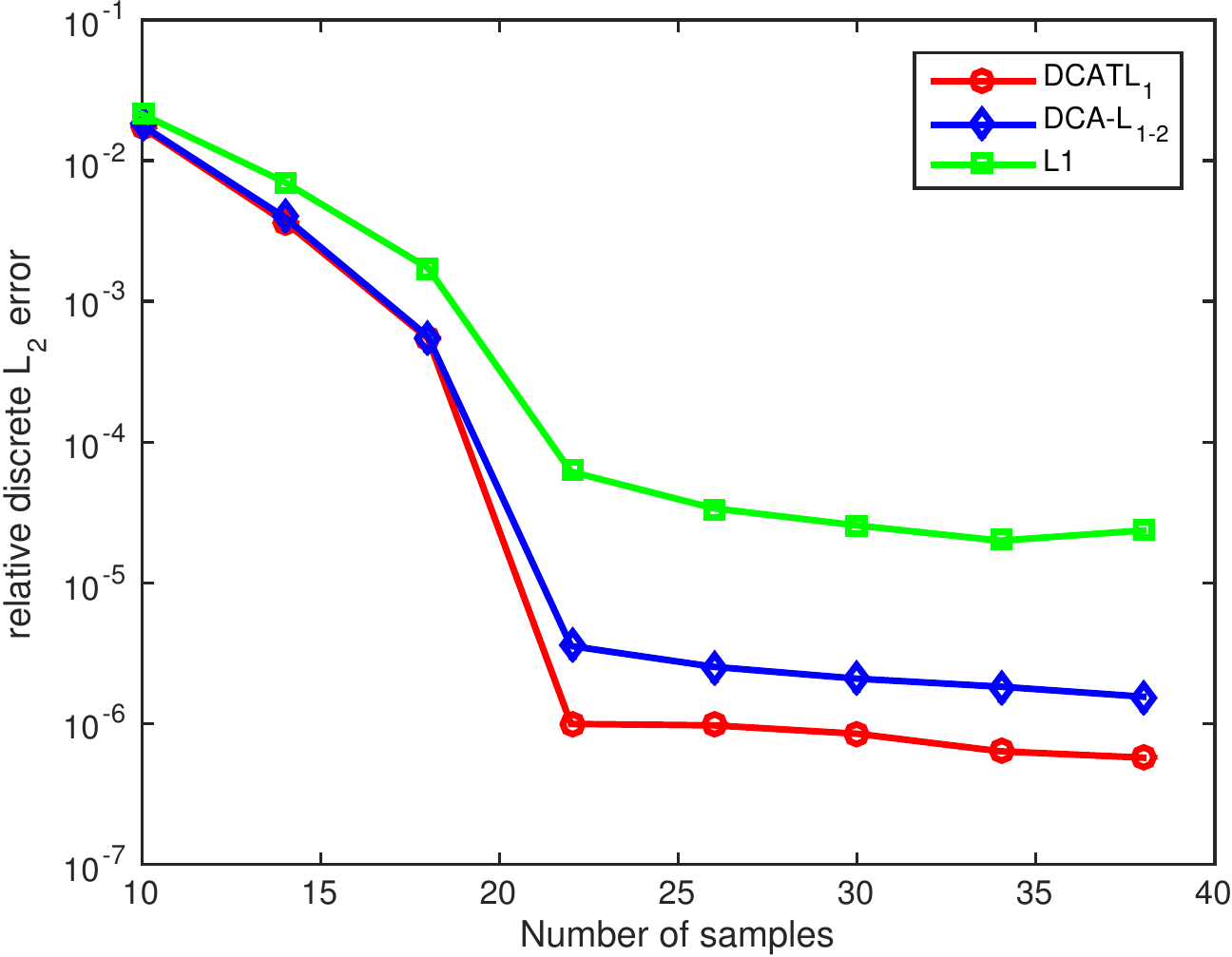}\quad
\includegraphics[width=0.48\textwidth,height=0.23\textheight]{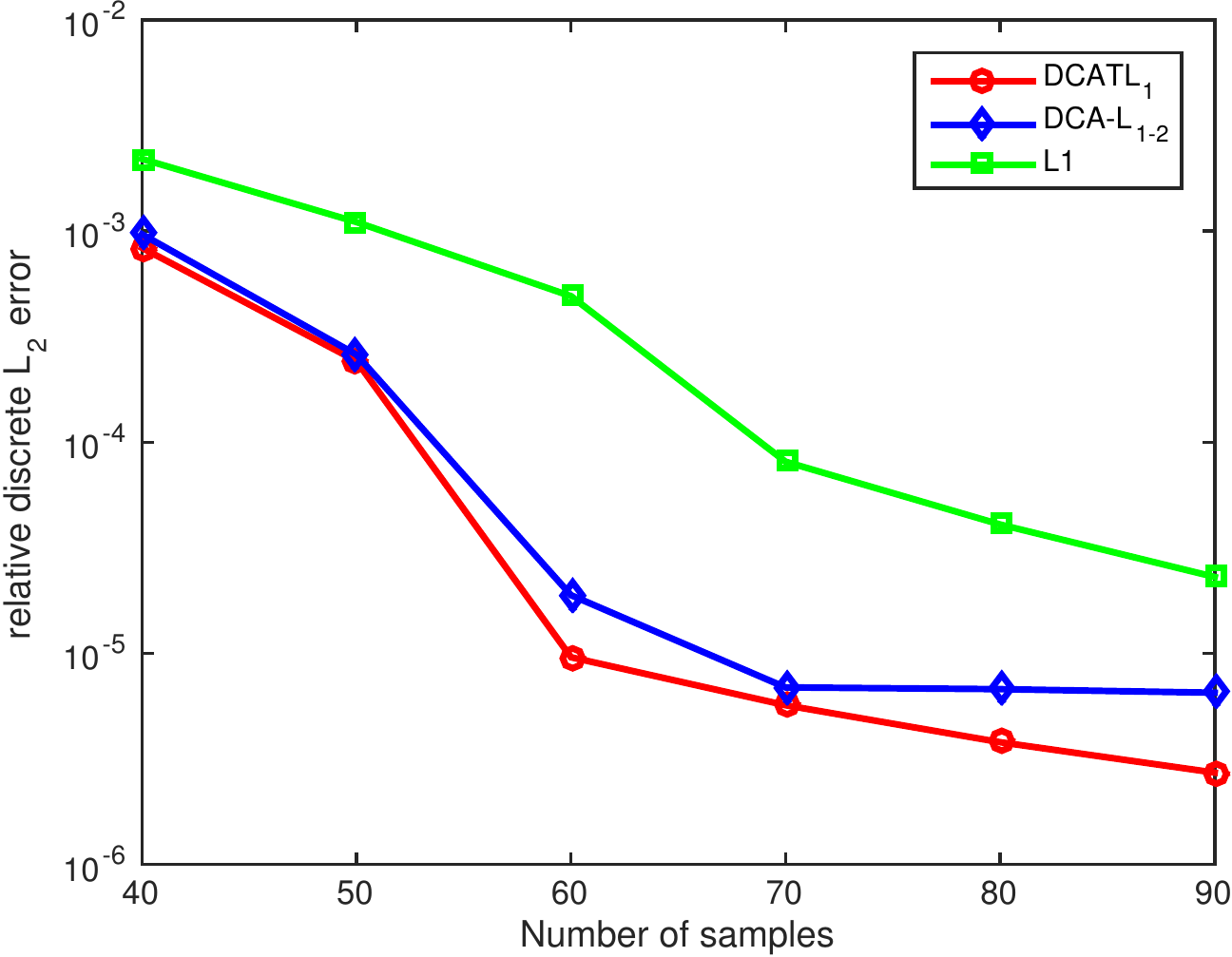}
\caption{Comparison of the relative discrete $\ell_2$ error in reconstructing $u(x=-6.5878, t=1)$ via standard $\ell_1$-minimization, $\ell_{1-2}$-minimization and TL1 minimization. Left: $d=2, k=20$; Right: $d=10, k=4$.}\label{fig:kdverrmodel}
\end{figure}

We finally draws the magnitudes of the gPC coefficients computed by
the three methods for both settings $d=2,k=20,N=231$ and $d=10,k=4,N=1001$ in Figure \ref{fig:kdvcoefmodel}. We see that TL1 minimization gives sparser solutions than $\ell_{1-2}$ and standard $\ell_1$ minimization.
\begin{figure}[!ht]
\centering
\includegraphics[width=0.48\textwidth,height=0.23\textheight]{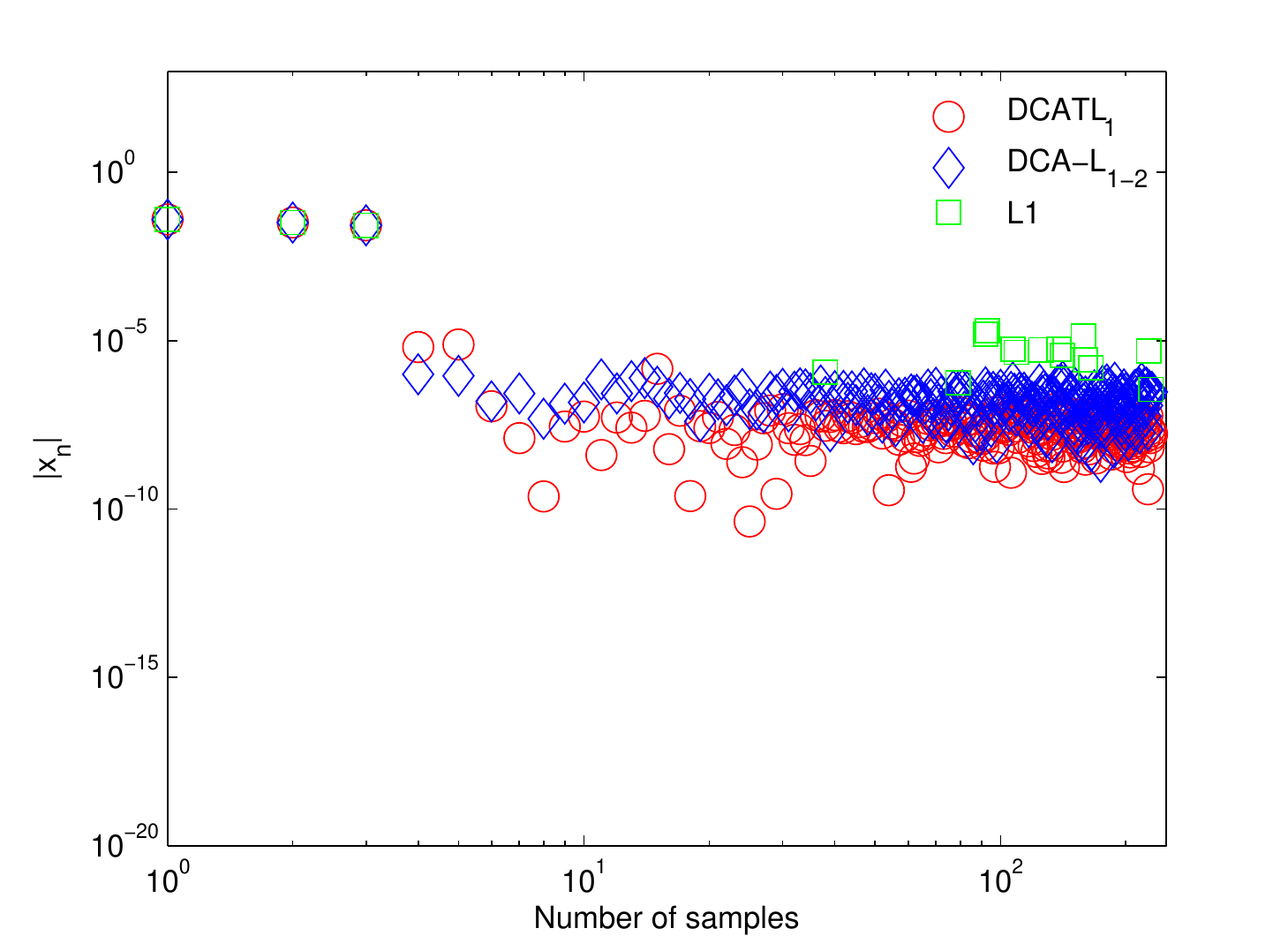}\quad
\includegraphics[width=0.48\textwidth,height=0.23\textheight]{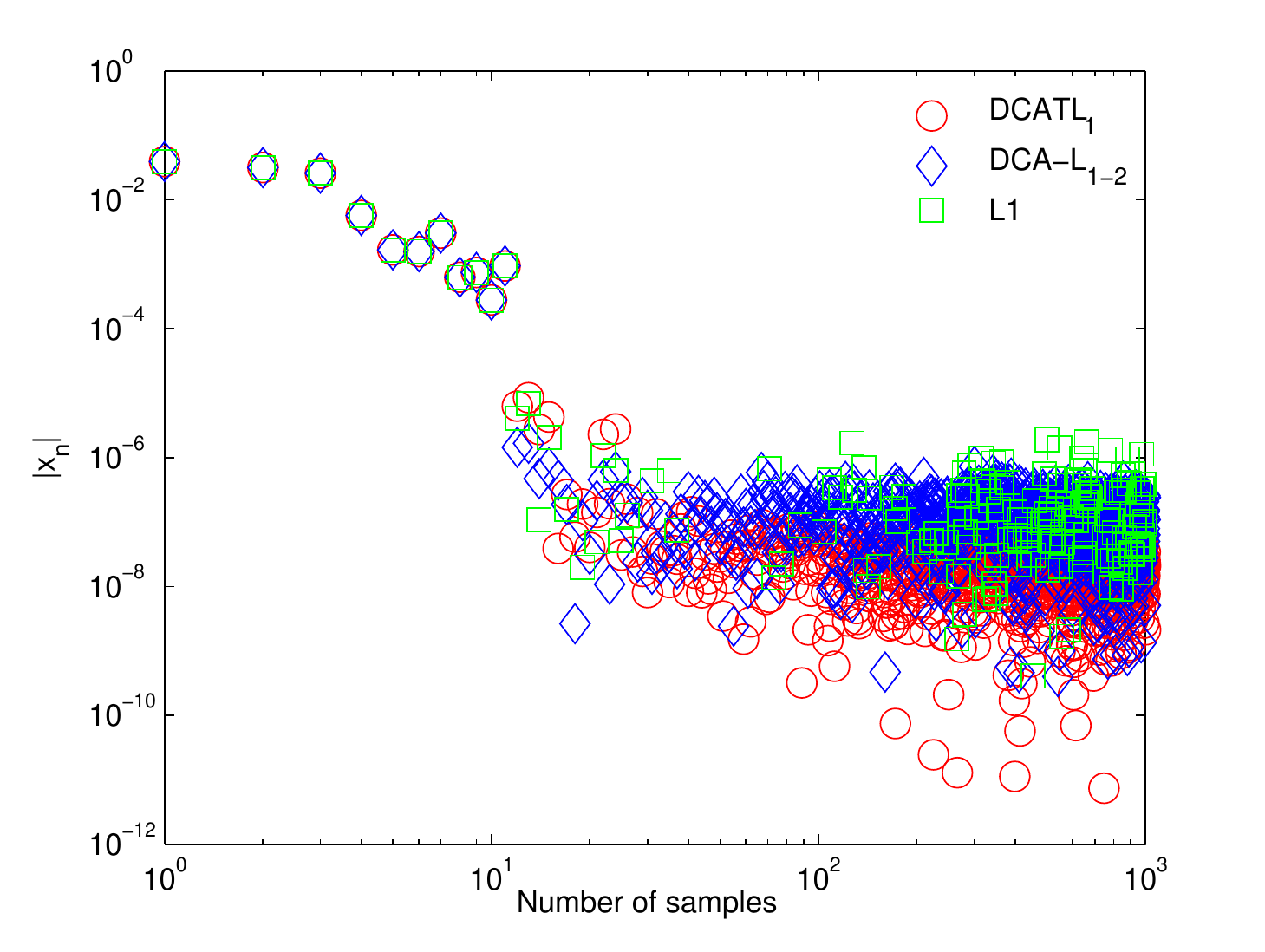}
\caption{The magnitude of the gPC coefficients computed with standard $\ell_1$-minimization, $\ell_{1-2}$-minimization and TL1 minimization for the
KdV equation. Left: $d=2, k=20,N=231$ and sample number $M=30$; Right: $d=10, k=4,N=1001$ and sample number $M=200$.}\label{fig:kdvcoefmodel}
\end{figure}
\section{Conclusion}

From the theoretical analysis of TL1 minimization, we conclude that under a more elegant RIP condition, TL1 method can reproduce the solution of $\ell_0$ minimization. If the target function is sparse, then we can obtain an exact solution, while if non-sparse, then we derive a stable approximation. Moreover, our theoretical analysis guarantees the recovery no matter the measurement data contain noise or not. Numerical experiments demonstrate that the TL1 minimization method is the most efficient method among all three selected methods (standard $\ell_1$, $\ell_{1-2}$ and TL1 minimization). It is worth noting that to recover non-sparse functions, we take an adaptive DCA-TL1 algorithm (Algorithm \ref{alg_2}), because we can not determine the optimal parameter $a$ before solving the problem. Finally, in this paper we only consider the noiseless recovery algorithm and noiseless numerical tests, the noisy recovery algorithm is left for the future study, even though the theoretical results are available.

\appendix
\section{Proof of Theorem \ref{thm:noisyrecover}}\label{apx:prof_noisyrecover}

To prove the Theorem \ref{thm:noisyrecover}  we follow the lines of \cite{Candes_2008Rip}.

\begin{proof}
Let $x$ be any feasible solution satisfying the constraint $b=Bx+e$ and $\hat{x}$ is the solution to the problem \eqref{eq:PaTL1mindenoise}, hence, we obtain
\begin{align}\label{eq:proof2eq1}
P_a(\hat{x})\leq P_a({x}).
\end{align}

We set $\hat{x}=x+h$ and decompose $h$ into a sum of s-sparse vectors $h_{T_0}, h_{T_1}, h_{T_2},\ldots$. Let $T_0$ corresponds to the locations of the $s$ largest coefficients of $x$; $T_1$ to the locations of the $s$ largest coefficients of $h_{T_0^c}$, $T_2$ to the locations of the $s$ largest coefficients of $h_{(T_0\cup T_1)^c}$, and so on. Also denote $T_{01}=T_0\cup T_1$. Then, \eqref{eq:proof2eq1} becomes
\begin{equation*}
P_a(x+h)\leq P_a(x).
\end{equation*}
Since $x = x_{T_0}+x_{T^c_0}$ and $h=h_{T_0}+h_{T_0^c}$, then
\begin{equation}\label{eq:Pa_ineq_r}
P_a(x_{T_0}+x_{T_0^c}+h_{T_0}+h_{T_0^c})\leq P_a(x_{T_0}+x_{T_0^c})=P_a(x_{T_0})+P_a(x_{T^c_0}).
\end{equation}
On the other hand, by Lemma \ref{lem:normtrieq}, one have
\begin{equation}\label{eq:Pa_ineq_l}
\begin{split}
P_a(x_{T_0}+x_{T_0^c}+h_{T_0}+h_{T_0^c})&=P_a(x_{T_0}+h_{T_0})+P_a(x_{T_0^c}+h_{T_0^c})\\
&=\sum_{i\in T_0}\rho_a(|x_{T_{0,i}}+h_{T_{0,i}}|)+\sum_{i\in T_0^c}\rho_a(|x_{T_{0^c,i}}+h_{T_{0^c,i}}|)\\
&\geq\sum_{i\in T_0}\rho_a(|x_{T_{0,i}}|)-\sum_{i\in T_0}\rho_a(|h_{T_{0,i}}|)\\
&\quad +\sum_{i\in T_0^c}\rho_a(|x_{T_{0^c,i}}|)-\sum_{i\in T_0^c}\rho_a(|h_{T_{0^c,i}}|) \\
&=P_a(x_{T_0})-P_a(h_{T_0})+P_a(h_{T_0^c})-P_a(x_{T_0^c}).
\end{split}
\end{equation}
From \eqref{eq:Pa_ineq_r} and \eqref{eq:Pa_ineq_l} we have
\begin{equation}\label{eq:hconstraint}
P_a(h_{T_0^c})\leq P_a(h_{T_0})+2P_a(x_{T_0^c}).
\end{equation}

The rest of the proof mainly contains two steps. Firstly, we set an upper bound on $\sum_{j\geq2}\|h_{T_j}\|_2$ and prove that $\|h_{(T_0\cup T_1)^c}\|_2$ is essentially bounded by $\|h_{(T_0\cup T_1)}\|_2$. Secondly, we show that $\|h_{(T_0\cup T_1)}\|_2$ is relatively small.

For the first step, from \cite{Zhang_2016DCATL1},we note that
\begin{align*}
\sum_{j\geq2}\|h_{T_j}\|_2\leq s^{-1/2}P_a(h_{T_0^c}),
\end{align*}
thus we obtain the useful estimate
\begin{align*}
\|h_{(T_0\cup T_1)^c}\|_2=\|\sum_{j\geq2}h_{T_j}\|_2\leq\sum_{j\geq2}\|h_{T_j}\|_2\leq s^{-1/2}P_a(h_{T_0^c}).
\end{align*}
Combined with \eqref{eq:hconstraint} we obtain
\begin{align*}
\|h_{(T_0\cup T_1)^c}\|_2\leq s^{-1/2}P_a(h_{T_0})+2s^{-1/2}P_a(x_{T_0^c}).
\end{align*}
By the inequality $(3.10)$ in \cite{Zhang_2016DCATL1}, we have
\begin{align}\label{eq:deltainequ1}
\|h_{(T_0\cup T_1)^c}\|_2\leq(1+\frac{1}{a})\|h_{T_{01}}\|_2+2s^{-1/2}P_a(x_{T_0^c}).
\end{align}
Since $Bh_{T_0\cup T_1}=Bh-\sum_{j\geq2}Bh_{T_j}$, it follows that
\begin{align*}
\|Bh_{T_0\cup T_1}\|_2^2=\langle Bh_{T_0\cup T_1}, Bh\rangle-\langle Bh_{T_0\cup T_1},\sum_{j\geq2}Bh_{T_j}\rangle.
\end{align*}
By the Cauchy-Schwarz inequality and the restricted isometry property, we have
\begin{align*}
|\langle Bh_{T_0\cup T_1}, Bh\rangle|\leq \|Bh_{T_0\cup T_1}\|_2\|Bh\|_2\leq2\varepsilon\sqrt{1+\delta_{2s}}\|h_{T_0\cup T_1}\|_2,
\end{align*}
moreover,
\begin{align*}
|\langle Bh_{T_i}, Bh_{T_j}\rangle|\leq\delta_{2s}\|h_{T_i}\|_2\|h_{T_j}\|_2, \ \text{for}\ i=0,1, j\geq 2
\end{align*}
and
\begin{align*}
\|h_{T_0}\|_2+\|h_{T_1}\|_2\leq\sqrt2\|h_{T_0\cup T_1}\|_2.
\end{align*}
Therefore
\begin{align*}
(1-\delta_{2s})\|h_{T_0\cup T_1}\|_2^2\leq\|Bh_{T_0\cup T_1}\|_2^2\leq\|h_{T_0\cup T_1}\|_2(2\varepsilon\sqrt{1+\delta_{2s}}+\sqrt2\delta_{2s}\sum_{j\geq2}\|h_{T_j}\|_2).
\end{align*}
Furthermore, we obtain
\begin{align*}
\begin{split}
\|h_{T_0\cup T_1}\|_2&\leq\frac{2\varepsilon\sqrt{1+\delta_{2s}}}{1-\delta_{2s}}+\frac{\sqrt2\delta_{2s}}{1-\delta_{2s}}\sum_{j\geq2}\|h_{T_j}\|_2 \\
&\leq\frac{2\varepsilon\sqrt{1+\delta_{2s}}}{1-\delta_{2s}}+\frac{\sqrt2\delta_{2s}}{1-\delta_{2s}}s^{-1/2}P_{a}(h_{T_0^c}).
\end{split}
\end{align*}
Set $\alpha=\frac{2\sqrt{1+\delta_{2s}}}{1-\delta{2s}}$ and $\beta=\frac{\sqrt2\delta_{2s}}{1-\delta_{2s}}$, hence
\begin{align*}
\begin{split}
\|h_{T_0\cup T_1}\|_2&\leq\alpha\varepsilon+\beta(P_a(h_{T_0})+2P_a(x_{T_0^c})) \\
&\leq\alpha\varepsilon+\beta(1+\frac{1}{a})\|h_{T_{01}}\|_2+2\beta s^{-1/2}P_a(x_{T_0^c}),
\end{split}
\end{align*}
It follows that
\begin{align}\label{eq:deltainequ2}
\|h_{T_{01}}\|_2\leq(1-\beta(1+\frac{1}{a}))^{-1}(\alpha\varepsilon+2\beta s^{-1/2}P_a(x_{T_0^c})).
\end{align}
Finally, by combining \eqref{eq:deltainequ1} and \eqref{eq:deltainequ2} we obtain that
\begin{align*}
\begin{split}
\|h\|_2&\leq\|h_{T_{01}}\|+\|h_{T_{01}^c}\|_2\\
&\leq(2+\frac{1}{a})(1-\beta(1+\frac{1}{a}))^{-1}\alpha\varepsilon+(4\beta(2+\frac{1}{a})(1+\frac{1}{a}))^{-1}+2)s^{-1/2}P_a(x_{T_0^c})\\
&=C_0s^{-1/2}P_a(x-x_s)+C_1\varepsilon.
\end{split}
\end{align*}
where we defined $C_0$ and $C_1$ as below
\begin{align}
\begin{split}
C_0&=\frac{6a\beta+2a+2\beta}{a-(a+1)\beta}\\
&=\frac{(6\sqrt2a-2a+2\sqrt2)\delta_{2s}+2a}{a-((\sqrt2+1)a+\sqrt2)\delta_{2s}},
\end{split}
\end{align}
\begin{align}\label{eq:TL1normalC}
\begin{split}
C_1&=\frac{(2a+1)\alpha}{a-(a+1)\beta}\\
&=\frac{2(2a+1)\sqrt{1+\delta_{2s}}}{a-((\sqrt2+1)a+\sqrt2)\delta_{2s}}.
\end{split}
\end{align}
\end{proof}

\bibliographystyle{plain}
\bibliography{sc-uq-DCATL1}

\end{document}